\documentclass[12pt,a4paper]{amsart}
\usepackage{mathtools}
\usepackage{xcolor}
\usepackage{amsmath,amstext,amssymb,amsthm,amsfonts}
\usepackage[all]{xy} \xyoption{poly}
\newcommand{\nc}{\newcommand}

\nc{\one}{\mbox{\bf 1}}
\nc{\invtensor}{\underset{\leftarrow}{\otimes}}
\nc{\const}{\operatorname{const}}

\nc{\ad}{\operatorname{ad}}

\nc{\tr}{\operatorname{tr}}
\nc{\diag}{\operatorname{diag}}
\nc{\zero}{\operatorname{zero}}

\nc{\tp}{\operatorname{top}}
\nc{\rank}{\operatorname{rank}}
\nc{\corank}{\operatorname{corank}}
\nc{\codim}{\operatorname{codim}}
\nc{\sdim}{\operatorname{sdim}}
\nc{\mult}{\operatorname{mult}}
\nc{\ds}{\operatorname{ds}}
\nc{\tail}{\operatorname{tail}}
\nc{\howl}{\operatorname{howl}}
\nc{\spn}{\operatorname{span}}
\nc{\defect}{\operatorname{defect}}
\nc{\Sym}{\operatorname{Sym}}
\nc{\sym}{\operatorname{sym}}
\nc{\id}{\operatorname{id}}
\nc{\Id}{\operatorname{Id}}
\nc{\Ree}{\operatorname{Re}}
\nc{\hi}{\operatorname{hi}}
\nc{\htt}{\operatorname{ht}}
\nc{\at}{\operatorname{at}}
\nc{\str}{\operatorname{str}}
\nc{\Iso}{\operatorname{Iso}}
\nc{\Ker}{\operatorname{Ker}}
\nc{\rker}{\operatorname{rKer}}
\nc{\im}{\operatorname{Im}}
\nc{\osp}{\mathfrak{osp}}
\nc{\sgn}{\operatorname{sgn}}
\nc{\F}{\operatorname{F}}
\nc{\Mod}{\operatorname{Mod}}
\nc{\DS}{\operatorname{DS}}
\nc{\Soc}{\operatorname{Soc}}
\nc{\coSoc}{\operatorname{coSoc}}
\nc{\Rad}{\operatorname{Rad}}
\nc{\Sch}{\operatorname{Sch}}
\nc{\Hom}{\operatorname{Hom}}
\nc{\End}{\operatorname{End}}
\nc{\supp}{\operatorname{supp}}
\nc{\Card}{\operatorname{Card}}
\nc{\Ann}{\operatorname{Ann}}
\nc{\Ind}{\operatorname{Ind}}
\nc{\Coind}{\operatorname{Coind}}
\nc{\wt}{\operatorname{hwt}}
\nc{\hwt}{\operatorname{wt}}
\nc{\arc}{\operatorname{arc}}
\nc{\Arc}{\operatorname{Arc}}
\nc{\ch}{\operatorname{ch}}
\nc{\sch}{\operatorname{sch}}
\nc{\mdim}{\operatorname{mdim}}
\nc{\Stab}{\operatorname{Stab}}
\nc{\Irr}{\operatorname{Irr}}
\nc{\Spec}{\operatorname{Spec}}
\nc{\Res}{\operatorname{Res}}
\nc{\Aut}{\operatorname{Aut}}
\nc{\Ext}{\operatorname{Ext}}
\nc{\ext}{\operatorname{ext}}
\nc{\Prec}{\operatorname{Prec}}
\nc{\Fract}{\operatorname{Fract}}
\nc{\gr}{\operatorname{gr}}
\nc{\deff}{\operatorname{def}}
\nc{\core}{\operatorname{core}}
\nc{\HC}{\operatorname{HC}}
\nc{\dpth}{\operatorname{dpth}}
\nc{\sw}{\operatorname{sw}}
\nc{\red}{\operatorname{red}}

\nc{\pos}{\operatorname{pos}}

\nc{\wdchi}{\widetilde{\chi}}
\nc{\wdH}{\widetilde{H}}
\nc{\wdN}{\widetilde{N}}
\nc{\wdM}{\widetilde{M}}
\nc{\wdO}{\widetilde{O}}
\nc{\wdR}{\widetilde{R}}

\nc{\wdV}{\widetilde{V}}

\nc{\wdC}{\widetilde{C}}

\nc{\pari}{\operatorname{dex}}
\nc{\atyp}{\operatorname{atyp}}

\nc{\Core}{\operatorname{Core}}
\nc{\pr}{\operatorname{pr}}

\nc{\Obj}{\operatorname{Obj}}
\nc{\Dglie}{\operatorname{{\mathcal D}glie}}
\nc{\Fin}{\operatorname{{\mathcal{F}\!\mathit{in}}}}

\nc{\Adm}{\operatorname{\mathcal{A}dm}}

\nc{\Sg}{{\cS(\fg)}}

\nc{\Ug}{{\cU(\fg)}}

\nc{\Sh}{{\cS(\fh)}}
\nc{\Uh}{{\cU(\fh)}}

\nc{\Zg}{{{\mathcal{Z}}(\fg)}}

\nc{\Vir}{{\mathcal{V}ir}}
\nc{\NS}{{\mathcal{N}S}}

\nc{\tZg}{{\widetilde{\mathcal Z}({\mathfrak g})}}
\nc{\Zk}{{\mathcal Z}({\mathfrak k})}

\nc{\Up}{{\mathcal U}({\mathfrak p})}
\nc{\Ah}{{\mathcal A}({\mathfrak h})}
\nc{\Ag}{{\mathcal A}({\mathfrak g})}
\nc{\Ap}{{\mathcal A}({\mathfrak p})}
\nc{\Zp}{{\mathcal Z}({\mathfrak p})}
\nc{\cR}{\mathcal R}
\nc{\cS}{\mathcal S}
\nc{\cP}{\mathcal P}
\nc{\cT}{\mathcal{T}}
\nc{\CC}{\mathcal C}
\nc{\cA}{\mathcal A}
\nc{\cE}{\mathcal E}
\nc{\cU}{\mathcal U}
\nc{\cZ}{\mathcal Z}
\nc{\cN}{\mathcal N}
\nc{\cM}{\mathcal M}
\nc{\cL}{\mathcal L}
\nc{\cF}{\mathcal F}
\nc{\fg}{\mathfrak g}
\nc{\cB}{\mathcal{B}}

\nc{\fo}{\mathfrak o}

\nc{\fz}{\mathfrak z}

\nc{\CO}{\mathcal O}
\nc{\CR}{\mathcal R}

\nc{\cK}{\mathcal{K}}
\nc{\cW}{\mathcal{W}}
\nc{\bM}{\mathbf{M}}
\nc{\bL}{\mathbf{L}}
\nc{\bN}{\mathbf{N}}

\nc{\zq}{\mathpzc q}

\nc{\fl}{\mathfrak l}
\nc{\fn}{\mathfrak n}
\nc{\fm}{\mathfrak m}
\nc{\fp}{\mathfrak p}
\nc{\fh}{\mathfrak h}
\nc{\ft}{\mathfrak t}
\nc{\fk}{\mathfrak k}
\nc{\fb}{\mathfrak b}
\nc{\fs}{\mathfrak s}
\nc{\fr}{\mathfrak r}
\nc{\fB}{\mathfrak B}

\nc{\vareps}{\varepsilon}
\nc{\varesp}{\varepsilon}
\nc{\veps}{\varepsilon}

\nc{\fsl}{\mathfrak{sl}}
\nc{\fgl}{\mathfrak{gl}}
\nc{\fso}{\mathfrak{so}}
\nc{\fosp}{\mathfrak{osp}}
\nc{\fsp}{\mathfrak{sp}}
\nc{\fq}{\mathfrak q}
\nc{\fsq}{\mathfrak{sq}}
\nc{\fpsq}{\mathfrak{psq}}

\nc{\ftg}{\hat{\fg}}
\nc{\ftn}{\hat{\fn}}
\nc{\fth}{\hat{\ft}}
\nc{\ftb}{\hat{\fb}}
\nc{\hrho}{\hat{\rho}}

\nc{\hsl}{\hat{\fsl}}
\nc{\fpo}{\mathfrak{po}}
\nc{\dirlim}{\underset{\rightarrow}{\lim}\,}
\nc{\nen}{\newenvironment}
\nc{\ol}{\overline}
\nc{\ul}{\underline}
\nc{\ra}{\rightarrow}
\nc{\lra}{\longrightarrow}
\nc{\Lra}{\Longrightarrow}
\nc{\bo}{\bar{1}}
\nc{\Lla}{\Longleftarrow}

\nc{\Llra}{\Longleftrightarrow}

\nc{\thla}{\twoheadleftarrow}

\nc{\lang}{(}
\nc{\rang}{)}

\nc{\hra}{\hookrightarrow}

\nc{\iso}{\overset{\sim}{\lra}}

\nc{\ssubset}{\underset{\not=}{\subset}}

\nc{\vac}{|0\rangle}

\nc{\simka}{{\ \scriptscriptstyle _{{\sim}}^\text{\tiny{k}}\ }}

\nc{\Thm}[1]{Theorem~\ref{#1}}
\nc{\Prop}[1]{Proposition~\ref{#1}}
\nc{\Lem}[1]{Lemma~\ref{#1}}
\nc{\Cor}[1]{Corollary~\ref{#1}}
\nc{\Conj}[1]{Conjecture~\ref{#1}}
\nc{\Claim}[1]{Claim~\ref{#1}}
\nc{\Defn}[1]{Definition~\ref{#1}}
\nc{\Exa}[1]{Example~\ref{#1}}
\nc{\Rem}[1]{Remark~\ref{#1}}
\nc{\Note}[1]{Note~\ref{#1}}
\nc{\Quest}[1]{Question~\ref{#1}}
\nc{\Hyp}[1]{Hypoth\`ese~\ref{#1}}
\nen{thm}[1]{\label{#1}{\bf Theorem.\ } \em}{}
\nen{prop}[1]{\label{#1}{\bf Proposition.\ } \em}{}
\nen{lem}[1]{\label{#1}{\bf Lemma.\ } \em}{}
\nen{cor}[1]{\label{#1}{\bf Corollary.\ } \em}{}
\nen{conj}[1]{\label{#1}{\bf Conjecture.\ } \em}{}

\nen{claim}[1]{\label{#1}{\bf Claim.\ } \em}{}

\nen{defn}[1]{\label{#1}{\bf Definition.\ } }{}
\nen{exa}[1]{\label{#1}{\bf Example.\ } }{}
\nen{rem}[1]{\label{#1}{\em Remark.\ } }{}
\nen{note}[1]{\label{#1}{\em Note.\ } }{}
\nen{exer}[1]{\label{#1}{\em Exercise.\ } }{}
\nen{sket}[1]{\label{#1}{\em Sketch of proof.\ } }{}
\nen{quest}[1]{\label{#1}{\bf Question.\ } \em}{}

\nen{hyp}[1]{\label{#1}{\bf Hypoth\`ese.\ } \em}{}
\begin{document}
\setcounter{section}{-1}
\setcounter{tocdepth}{1}

\title[ ]
{On modified  extension graphs of a fixed atypicality}

\author{Maria Gorelik }

\address[]{Dept. of Mathematics, The Weizmann Institute of Science,Rehovot 76100, Israel}
\email{maria.gorelik@weizmann.ac.il}

\begin{abstract}
In this paper we  study extensions between finite-dimensional simple modules
over classical Lie superalgebras $\fgl(m|n), \osp(M|2n)$ and $\fq_m$. 
We consider a simplified version of
the extension graph which is produced from the $\Ext^1$-graph by identifying
 representations obtained by parity change and removal of the loops.
We give a necessary condition for a pair of vertices to be connected  and show that 
 this condition is sufficient in most of the cases. This condition implies
that the image of a finite-dimensional simple module
 under the Duflo-Serganova functor  has indecomposable  isotypical components.
 This yields semisimplicity of Duflo-Serganova functor
for $\Fin(\fgl(m|n))$ and for $\Fin(\osp(M|2n))$.
\end{abstract}

\maketitle

\section{Introduction}
Let $\CC$ be a category of representations of
a Lie superalgebra $\fg$ and $\Irr(\CC)$ be the set of isomorphism classes
of simple modules in $\CC$.  Assume that the modules in $\CC$ are of  finite length.
\footnote{ this  can be replaced by existence of local composition series  constructed in~\cite{DGK}.}
In many examples the extension graph of $\CC$ is bipartite, i.e. there exists a map
$\pari:\Irr(\CC)\to \mathbb{Z}_2$ such that

\begin{enumerate}
\item[$ \,$]
(Dex1)$\ \ \ $ $\Ext^1_{\CC}(L_1,L_2)=0\ $ if
$\ \ \pari(L_1)=\pari(L_2)$.
\end{enumerate}

In what follows $\Fin(\fg)$ stands for the full subcategory 
of the category of finite-dimensional $\fg$-modules consisting of the modules which are competely reducible over
 $\fg_{\ol{0}}$.
In this paper we consider the following examples:

\begin{itemize}
\item[(KM)]
$\fg=\fgl(m|n),\osp(M|2n)$ 
and $\CC=\Fin(\fg)$;

\item[$(\fq;\frac{1}{2})$]
$\CC=\Fin(\fq_m)_{1/2}$ which is the full subcategory  of $\Fin(\fq_m)$ consisting of the modules
with  ``half-integral'' weights;

\item[$(\fq;\CC)$]
$\CC$ is a certain full subcategory  of $\Fin(\fq_m)$, see~\ref{Dex12}.
\end{itemize}

By~\cite{HW},\cite{GH} in the case (KM) the map $\pari$ can be chosen to be ``compatible''  with the  Duflo-Serganova functors  introduced in~\cite{DS}:
\begin{enumerate}
\item[$ \,$]
(Dex2)$\ \ \ $  one has $\ [\DS_x(L):L']=0\ $ if $\ \pari(L)\not=\pari(L')$.
\end{enumerate}
Note that in  (Dex2) we have to choose $\pari$ on $\Irr(\CC)$ and on 
$\Irr(\DS_x(\CC))$.  Another example when (Dex1) and (Dex2)
hold is   $\CC=\Fin(\fg)$ for the exceptional Lie superalgebras ($D(2|1;a)$,
 $F(4)$ and $G(3)$) and when $\CC$ is the full subcategory 
of integrable modules  in the category $\CO(\fgl(1|n)^{(1)})$,
see~~\cite{Germoni1},\cite{Germoni2},\cite{Lilit} and~\cite{Gdex} for exceptional cases and~\cite{GSaff} for $\fgl(1|n)^{(1)}$. Note that $\Fin(\fg)$  coincides with the full subcategory 
of integrable modules  in the category $\CO(\fg)$ if $\dim\fg<\infty$.
This  suggests  the following conjecture: if $\fg$ is  a Kac-Moody superalgebra, 
then the full subcategory 
of integrable modules  in the category $\CO(\fg)$ admits a map $\pari$ satisfying (Dex1) and (Dex2).

In (KM) case the map $\pari$ implicitly appeared in~\cite{Skw},
Theorem 5.7.
The original motivation for this project was to study complete reducibility
of $\DS_x(L)$, where 
 $L$ is a finite-dimensional simple $\fg$-module. 
 Clearly, the existence of
  $\pari$  satisfying (Dex1) and  (Dex2) implies complete reducibility
  of $\DS_x(L)$ for each $L\in\Irr(\CC)$. For the categories $\Fin(\fg)$,
where $\fg$ is a finite-dimensional Kac-Moody superalgebra, a suitable map $\pari$
satisfying (Dex1)
was described in~\cite{Gdex} and  the condition
(Dex2) was verified in~\cite{HW},\cite{GH};  this gives the
complete reducibility  of $\DS_x(L)$ for each simple finite-dimensional
module $L$ over a finite-dimensional Kac-Moody superalgebra.
Unexpectedly, it turns out that the complete reducibility holds in the cases
$(\fq;\frac{1}{2})$ and $(\fq;\CC)$
even though (Dex2) does not hold. Below we will explain how the complete reducibility 
can be obtained in the absence of the property (Dex2).

Denote by $\Pi$ the parity change functor and by $L(\lambda)$ a simple $\fg$-module of
the highest weight $\lambda$; we set
\begin{equation}\label{extintro}
\ext_{\fg}(\lambda;\nu)=\left\{\begin{array}{l}
\dim\Ext^1_{\CO} (L(\lambda),L(\nu))\ \ \ \ \ \ \ \ \ \ \ \ 
\text{ if $L(\nu)$ is $\Pi$-invariant}\\
\dim\Ext^1_{\CO} (L(\lambda),L(\nu))+\dim\Ext^1_{\CO}  (L(\lambda),\Pi L(\nu)) \ \ 
\text{ otherwise.}\end{array}\right.
\end{equation}

Let $(\CC;\ext)$ be the graph with 
the set of vertices $\Irr(\CC)$ modulo the involution defined by $\Pi$, with $L(\lambda)$ and $L(\nu)$ 
connected by $\ext(\lambda;\nu)$ edges 
if $\lambda\not=\nu$. 
This graph can be obtained from the usual $\Ext^1$-graph in two steps:  factoring
modulo the involution followed by deleting the loops. The graph obtained by factoring  modulo the involution does not have loops if $\fg$ is a Kac-Moody superalgebra and $\CC\subset\CO(\fg)$; in
 the $\fq_m$-case
there is at most one loop around each vertex 
and the vertices with loops
correspond to the weights having at least one zero coordinate, see Theorem 3.1 of~\cite{GNS}.
In many cases the involution defined by $\Pi$ permute the components of the
$\Ext^1$-graph of $\CC$, so  this graph is  isomorphic to two disjoint copies
of $(\CC;\ext)$. It is easy to see that this property holds if $\fg$ is a Kac-Moody superalgebra and $\CC\subset\CO(\fg)$.  The 
category $\Fin(\fq_m)_{1/2}$ was described in~\cite{BD}; their results yield the above property. This property does not hold for the atypical 
integral blocks in $\Fin(\fq_m)$.

Our main result is formulated in terms 
  of ``arc/arch diagrams'' which  were used  in~\cite{HW},\cite{EAS},\cite{GS} and in~\cite{DSnew}; for $\fq_n$-case a modification of these diagrams is used in~\cite{GqDS}.

{\em Theorem A. Let $\fg=\fgl(m|n),\osp(M|2n)$ or $\fq_m$.
If there exists a non-split extension between two non-isomorphic finite-dimensional simple modules, then  the weight diagram of one of these modules can be obtained from the weight diagram of the other module by moving one or two symbols $\times$ along one of the arches. For $\fg=\fgl(m|n),\osp(M|2n)$ and for half-integral weights
of $\fq_m$ this condition is sufficient.}

The above condition is not sufficient for integral weights
of $\fq_m$, but it is sufficient if the modules are ``large enough'', see~\Cor{corThm1}.

Theorem A implies that $\pari$, introduced by the formula~(\ref{paridef}) below, satisfies
 (Dex1) in all cases we consider.
The aforementioned description of  $\DS_x(L)$  implies that the arch diagram 
corresponding to  a subquotient of $\DS_x(L)$ can be obtained 
from the arch diagram of $L$ by sequential removal of several maximal arches.
This, together with Theorem A, implies that any  extension between
non-isomorphic simple subquotients of $\DS_x(L)$ splits. This gives the
complete reducibility of $\DS_x(L)$ in the cases (KM), ($\fq;\frac{1}{2}$),
 ($\fq;\CC$) and the fact that
in the $\fq_m$-case any indecomposable submodule of $\DS_x(L)$ is ``isotypical''
in the sense of~\ref{isotypical}.

\subsection{Remark} 
It was observed by
 Alex Sherman, our results  imply the following: each 
block of atypicality $k$ \footnote{$ \ $ In contrast to the Kac-Moody case, the definition of the degree of atypicality in the $\fq_n$-case
admits several variations; for example, our degree of atypicality is the integral part of the degree appearing in~\cite{GqDS}.
 } in $\Fin(\fg)$ contains 
 a ``large'' Serre subcategory $\CC_+$ (described in~\ref{Alex}) such that the graph $(\CC_+;\ext)$
is isomorphic to $(\CC_{1/2}(k);\ext)$, where $\CC_{1/2}(k)$ is the half-integral block of atypicality $k$ in $\Fin(\fq_{2k})$. In order to illustrate
this observation, consider the simplest case $k=1$ which was studied in~\cite{Germoni1},\cite{Germoni2},\cite{Lilit},\cite{MM} and \cite{GNS}. We have
the following three types of $\ext$-graphs:
$$\xymatrix{&A^{\infty}_{\infty}: &\ldots&\ar[r]& \bullet\ar[r]\ar[l] &\bullet\ar[l]\ar[r]&\bullet\ar[r] &\bullet\ar[l]\ar[r] &\bullet\ar[l]\ar[r] 
& \ldots &
\\
&D_{\infty}: & &  &\bullet\ar[r]
&\bullet\ar[r]\ar[l]\ar[d] &\bullet
\ar[r]\ar[l]&\ldots& \\
& & & &  & \bullet\ar[u]&&&
\\
&A_{\infty}: & &  &\bullet\ar[r] &\bullet\ar[l]\ar[r] &\bullet\ar[l]\ar[r] 
& \ldots &
}\\
$$

The first graph corresponds to the blocks of 
atypicality one in $\fgl(m|n)$, $\osp(2m|2n)$
and some blocks of atypicality one in $F(4)$ and $D(2|1;a)$  for  $a\in\mathbb{Q}$; the second graph corresponds to the blocks of 
atypicality one in $\osp(2m+1|2n),\osp(2m|2n), G(3)$ and the rest of the blocks of atypicality one in
  $F(4)$ and $D(2|1;a)$.
The third graph corresponds to 
 the blocks of atypicality one for $\fq_m$; this graph is contained in the first two graphs.  The picture is much more complicated for $k>1$. For instance, 
 the vertices of the blocks of atypicality two in $\Fin(\fg)$ are enumerated
 by the integral pairs $(i,j)$ satisfying the following conditions:
 \begin{itemize}
 \item $i<j$ for $\fgl(m|n)$
\item $0<i<j$ or $i=j=0$ for $\osp(2m+1|2n)$, the integral  blocks for $\fq_m$
 and certain blocks for $\osp(2m|2n)$;
 \item $|i|<j$ or $i=j=0$ for the rest of the atypicality two blocks  for $\osp(2m|2n)$;
  \item $0<i<j$ for the half-integral blocks for $\fq_m$.
 \end{itemize}
 The last graph is an induced subgraph
\footnote{the induced subgraph is the graph with the set of vertices $B$ which includes all edges $\mu\to \nu$ for $\mu,\nu\in\cB$.} of all above graphs except, perhaps,  for
 the integral  blocks for $\fg=\fq_m$;  for the latter case
  the last graph is isomorphic to the induced subgraph 
  for the vertices $(i,j)$ with $1<i<j$.

\subsection{Methods}
For the cases (KM), ($\fq;\frac{1}{2}$)  the categories $\Fin(\fg)$ were studied in many papers including~\cite{MS},\cite{BS},\cite{BD},\cite{ChK},\cite{ES1},\cite{ES2} and  Theorem A can be deduced from 
 the results of these papers. The categories $\Fin(\fq_2),\Fin(\fq_3)$ were described
in~\cite{MM} and~\cite{GNS} respectively. 
In this paper
we obtain Theorem A using the approach of~\cite{MS},\cite{GNS}. It is not hard to reduce
the assertion to the case  when $\fg$ is one of the algebras
$\fgl(n|n),\osp(2n+t|2n),\fq_{2n+\ell}$ with $t=0,1,2$ and $\ell=0,1$
and  the simple modules have the  same central character
as the trivial module. We take $\fg$ as above and
denote by $\cB$ the set of $\lambda$s such
that $L(\lambda)$ finite-dimensional and have the  same central character
as the trivial module.  Our main tools are the functors $\Gamma^{\fg,\fp}_{\bullet}$ introduced 
  in~\cite{Penkov} (we  use the ``dual version'' 
  appeared in~\cite{GS}). For a parabolic subalgebra $\fp\subset\fg$
  the functor $\Gamma^{\fg,\fp}_{\bullet}: \Fin(\fp)\to \Fin(\fg)$
  is the derived functor of the functor which maps each finite-dimensional
  $\fp$-module to the maximal finite dimensional quotient of the induced module $\cU(\fp)\otimes_{\cU(\fq)}V$. We fix a ``nice chain''
  of Lie superalgebras 
  $\fg_{(0)}\subset \fg_{(1)}\subset \ldots\subset \fg_{(n)}\ $
where $\fg_{(i)}=\fgl(i|i)$ for $\fg=\fgl(n|n)$, $\fg_{(i)}=\osp(2i+t|2i)$,
for $\fg=\osp(2n+t|2k)$ with $t=0,1,2$ and $\fg_{(i)}=\fq_{2i+\ell}$ for
$\fg=\fq_{2n+\ell}$ with $\ell=0,1$. For each $s$ the algebra 
${\fp}_{(s)}:=(\fg_{(s-1)}+\fb)\cap \fg_{(s)}$ is a parabolic subalgebra
in $\fg_{(s)}$. For $\fp:=\fp_{(n)}$
the multiplicities 
 $K^i(\lambda;\nu):=[\Gamma^{\fg,\fp}_i(L_{\fp}(\lambda)):L_{\fg}(\nu)]$  were 
 computed in~\cite{MS},\cite{GS},\cite{PS1},\cite{PS2}. 
 We will present the corresponding Poincar\'e polynomials 
 $K^{\lambda,\nu}(z):=\sum_i K^i(\lambda;\nu) z^i$ in terms of the arch diagram. The same Poincar\'e polynomials appear in the character formulae
 obtained in~\cite{GS},\cite{ChK}, \cite{SZq} and~\cite{GH2} (in particular,
 the arch diagrams in $\fq_m$-case are similar to the diagrams appeared in~\cite{SZq}, 3.3).  The multiplicites
 $$K^i_{(s)}(\lambda;\nu):=
 [\Gamma^{\fg_{(s)}+\fh,\fp_{(s)}+\fh}_i(L_{\fp_{(s)}+\fh}(\lambda)):L_{\fg}(\nu)]$$
 can be easily expressed in terms of  $K^i(\lambda;\nu)$
 computed for $\fg_{(i)}$. Set
 $k_0(\lambda;\nu):=\max_s K^0_{(s)}(\lambda;\nu)$. It turns out that
 $k_0(\lambda;\nu)\not=0$ implies that 
 the weight diagram of $\lambda$ can be obtained from the weight diagram of $\nu$ by moving one or two symbols $\times$ along one of the arches.
For $\fg=\fgl(n|n),\osp(2n+t|2n)$ the inequality
$k_0(\lambda;\nu)\not=0$  forces $k_0(\lambda;\nu)=1$
and $\pari(\lambda)\not=\pari(\nu)$ for the grading $\pari$ given by the formula~(\ref{paridef}).
For $\fg=\fq_m$-case the same hold if $\nu$ does not have zero coordinates.

 It is easy to see that  $\ext(\lambda;\nu)\leq k_0(\lambda;\nu)$
for $\lambda,\nu\in\cB$  with $\nu<\lambda$.  This gives
 the first claim of Theorem A for the case when the highest weights of the modules
 lie in $\cB$.  In~\Cor{corExtmain} we show that
$\ext(\lambda;\nu)=k_0(\lambda;\nu)$ for  $\lambda,\nu\in\cB$  with 
$\nu<\lambda$ except for the case
$\fg=\fq_m$ and  $\lambda$ has a coordinate which equals to  
 to $0$ and to $1+\ell$. This gives
 the second  claim of Theorem A for the case when the highest weights of the modules
 lie in $\cB$. The proof of the formula $\ext(\lambda;\nu)=k_0(\lambda;\nu)$ 
 is based on the fact that for each $s$ the radical of
 the maximal finite dimensional quotient of the induced module $\cU(\fp)\otimes_{\cU(\fq)}L_{\fp_{(s)}}(\lambda)$ is semisimple
 (the chain   $\fg_{(0)}\subset \fg_{(1)}\subset \ldots\subset \fg_{(n)}\ $ is
chosen so that this property holds).

\subsection{Content of the paper}
In Section~\ref{sectapp} we describe some background information about $\Ext^1$ and
 $\ext$; we obtain the formula  $\ext(\lambda;\nu)\leq k_0(\lambda;\nu)$
 and establish the formula $\ext(\lambda;\nu)=k_0(\lambda;\nu)$
 under certain assumption, see~\Cor{corgraphs}.
  In Section~\ref{arches} we introduce the language of ``arch diagrams'
 for the weights in $\cB$.
In Section~\ref{Section3} we deduce from the results of~\cite{PS1},\cite{PS2},
\cite{MS} and~\cite{GS} a description of
 the Poincar\'e polynomials $K^{\lambda,\nu}(z)$ in terms of the arch diagram.
 
In Section~\ref{sectiondex} we introduce the $\mathbb{Z}_2$-grading $\pari$
and show that $K^{\lambda,\nu}(z)$ has nice properties with respect to this grading.
Then we compute $\ext(\lambda;\nu)$ for $\lambda,\nu\in\cB$ under the assumption
that for $\fg=\fq_m$ the weight $\lambda$ does not have  coordinates
equal  to $0, 1$ and $1+\ell$. In~\ref{lambdanuP+} we establish Theorem A
by reducing the 
computations of $\ext(\lambda;\nu)$ for $\lambda,\nu\in P^+(\fg)$  
to the case $\lambda,\nu\in\cB$. In~\ref{Dex12} we 
define $\CC$ for the case $(\fq;\CC)$ and discuss
the conditions (Dex1), (Dex2) in various cases. Finally, in Remark~\ref{Extext} we discuss the connection between the $\ext$-graph
and $\Ext^1$-graph.

\subsubsection{}
This paper has a considerable overlap (the cases of $\fgl$ and $\fosp$) with the unpublished preprint~\cite{Gdex}.
This paper includes $\fq_n$-case whereas~\cite{Gdex} includes the case of exceptional
Lie superalgebras. In~\cite{Gdex} we studied the case of finite-dimensional Kac-Moody superalgebras;
in this case both (Dex1) and (Dex2) hold.

\subsection{Acknowledgments}  The author was supported by  ISF Grant 1957/21. The author is grateful to  N.~Davidson, V.~Hinich,  V.~Serganova and A.~Sherman for numerous 
helpful discussions.

\subsection{Index of definitions and notation} \label{sec:app-index}
Throughout the paper the ground field is $\mathbb{C}$; 
$\mathbb{N}$ stands 
 for the set of non-negative integers. We denote by $\Pi$ the parity change functor. 
We will use the notation $\Soc N$, $\Rad N$, $\coSoc N$ for the socle,
the radical and the cosocle of a module $N$ (recall that
 $\Soc N$ is the sum of simple submodules, $\Rad N$ is the intersection of maximal submodules and $\coSoc N:=N/Rad N$). Throughout the paper  $\equiv$ will be always used for the equivalence modulo $2$.

\begin{center}
\begin{tabular}{lcl}
$\fg$, $\ft$, (Ass$\ft$) & & \ref{axiomsH12} \\

$\CO(\fg),\ \Fin(\fg),\ C_{\lambda}, M(\lambda), L(\lambda), P^+(\fg)$ & &  \ref{Mlambda}\\

$[N:L]$ & & \ref{isotypical}\\

$\cN(\lambda;\nu;m),\ m(\fg,\fp;\lambda;\nu)$ & & \ref{cNm}\\
$\Gamma_i^{\fg,\fp},\ K^j(\lambda;\nu)$  & & \ref{Gammagp}\\
$\fg_{(s)},\ \fp_{(s)},\ \fh_{(s)},\ \ft_{(s)},\ 
\fh_{(s)}^{\perp}, \ \ft_{(s)}^{\perp}$ & & \ref{zii}\\
assumptions (A), (B) & & \ref{AiB}\\
 $K^i_{(s)}(\lambda;\nu)$, $\ext_{(s)}(\lambda;\nu)$  & & \ref{defKjs}\\
$s(\lambda;\nu)$, $k(\lambda;\nu)$, graph $G(\ft^*;K^0)$, $B(\lambda)$, $K^i$-stable & &  \ref{graphtK0}\\
triangular decompositions, $\rho$ & & \ref{tria}\\
 $\cB_0,\cB_{1/2}$, $\cB$, $\lambda_i$ & & \ref{lambda+rho}\\
arch diagram & & \ref{arcs}\\
$\fg_{(i)}$ for $\fg=\fgl(n|n),\osp(2n+t|2n),\fq_{2n+\ell}$; $\tail\lambda$
 & & \ref{gp}\\
$(g)_p^q,\ \ (g)_{0,0}^{p,q}\ $, $K(\frac{f}{g})$ & & \ref{fgreserved}\\
$\pari(\lambda), \pari(\lambda;\nu),\ \tail(\lambda;\nu)$ & & (\ref{paridef})\\
\end{tabular}
\end{center}

\section{Useful facts about $\ext(\lambda;\nu)$}\label{sectapp}

\subsection{}
\begin{lem}{lemExt}
Let $A$ be an associative superalgebra.
\begin{enumerate}
\item
If $N$ is an $A$-module  with a semisimple radical and a simple cosocle 
$L'$, then 
$$\dim\Hom(L,N)\leq  \dim \Ext^1(L',L).$$
for any simple $A$-module $L\not\cong L'$.

\item Let $L_1,\ldots, L_s,L'$ be simple non-isomorphic $A$-modules
and $m_1,\ldots,m_s$ be non-negative integers satisfying
$m_j\leq \dim \Ext^1(L',L_j)$.
There exists an $A$-module $N$ with 
$$\coSoc N\cong L',\ \ \ 
\Rad N\cong \oplus_j L_j^{\oplus m_j}.$$
\end{enumerate}
\end{lem}
\begin{proof}
Consider any exact sequence of the form
\begin{equation}\label{ExtN}
0\to L^{\oplus m}\overset{\iota}{\longrightarrow} N'
\overset{\phi}\longrightarrow L'\to 0.\end{equation}
For each $i=1,\ldots,m$ let
$p_i: L^{\oplus m}\to L$ be the projection to the $i$th component
and let $\theta_i: L\to L^{\oplus m}$ be the corresponding embedding $p_i\theta_i=Id_L$.
Consider a commutative diagram 
\begin{equation}\label{commdia}
\xymatrix{& 0\ar[r] & L^{\oplus m} \ar^{\iota}[r]\ar^{p_i}[d]&
N'\ar^{\phi}[r]\ar^{\psi_i}[d]&L'\ar[r]\ar^{Id}[d]& 0&\\
& 0\ar[r] & L\ar^{\theta_i}[u]
 \ar^{\iota_i}[r]&
M^i\ar^{\phi_i}[r]&L'\ar[r]& 0&
}
\end{equation}
where $\psi_i: N'\to M^i$ is a surjective map with 
$\Ker\psi_i=\iota(\Ker p_i)$. 
The bottom line of this diagram
is an element of $\Ext^1(L',L)$, which we denote by $\Phi_i$.

Assume that
$m>\dim \Ext^1(L',L)$. Then $\{\Phi_i\}_{i=1}^m$ are linearly dependent 
and we can assume that $\Phi_1=0$. This means that $\Phi_1$ splits, so
there exists a projection $\tilde{p}: M^1\to L$ with
$\iota_1\tilde{p}=Id_L$. Then $\tilde{p}\circ\psi_1\circ\iota\circ\theta_1=\Id_L$, so 
$ \tilde{p}\circ\psi_1: N'\to L$ is surjective. Therefore
$[\coSoc N':L]\not=0$, that is 
 $\coSoc N'\not\cong L'$.

Now take $N$ as in (i).  Let $N'$ be the quotient of $N$ by the sum of
all simple submodules which are not isomorphic to $L$.
One has 
$$\coSoc N'=\coSoc N\cong L',\ \ \ \ \Rad N'\cong L^{\oplus m}$$
where $m=\dim \Hom(L,N)$. By above, 
$m\leq \dim \Ext^1(L',L)$; this gives (i).

For (ii)  let $\{\Phi^{(j)}_i\}_{i=1}^{m_j}$ be linearly independent  elements in  $\Ext^1(L',L_j)$:
$$
\Phi^{(j)}_i:\ \ \ \ \ 0\to L_j\overset{\iota_i}{\longrightarrow} M^i\overset{\phi_i}{\longrightarrow} L'\to 0.$$
Consider the exact sequence 
$$0\to \oplus_j L^{\oplus m_j}\longrightarrow  \oplus_j\oplus_i M_i\to (L')^{\oplus \sum_{j=1}^ s m_j}\to 0.$$
Let $diag(L')$ be the diagonal copy of $L'$ in $(L')^{\oplus \sum_{j=1}^ s m_j}$ and
let $N$ be the preimage of $diag(L')$
in $\oplus_j\oplus_i M^i$. This gives the exact sequence 
\begin{equation}\label{ExtN}
0\to \oplus_j L^{\oplus m_j}\overset{\iota}{\longrightarrow} N
\overset{\phi}\longrightarrow L'\to 0\end{equation}
and the commutative diagrams similar to~(\ref{commdia}).
Assume that $N_1\subsetneq N$ is a  submodule of $N$ satisfying
$\phi(N_1)\not=0$.  
Since $\Ker \phi=Im\ \iota$ is completely reducible we have
$$\Ker\phi=(\Ker \phi\cap N_1)\oplus N_2$$
where $N_2\not=0$ is completely reducible. Then $N=N_1\oplus N_2$ and thus
 $N$ can be written as $N=L\oplus N_3$ where $L\subset N_2$
is simple, We can assume that $L\cong L_1$.
Changing the basis in the span of $\{\Phi^{(1)}_i\}_{i=1}^{m_j}$,
we can assume that 
$\iota(\Ker p_1)\subset N_3$. Since
$\Ker\psi_1=\iota (\Ker p_1)$, 
 $\psi_1(L)$ is a non-zero submodule of $M_1$, so
the exact sequence $\Phi_1$ splits, a contradiction.

Hence for each $N_1\subsetneq N$ one has $\phi(N_1)=0$ that is 
 $\Rad N=\Ker\phi$ as required.
\end{proof}

\subsection{Notation and assumptions}\label{axiomsH12}
Let $\fg$ be a Lie superalgebra of at most countable dimension
 with a finite-dimensional even subalgebra $\ft$
satisfying


(Ass$\ft$) $\ft$ acts diagonally on $\fg$ and $\fg_{\ol{0}}^{\ft}=\ft$.

We set $\fh:=\fh^{\ft}$ and choose $h\in\ft$ satisfying

\begin{equation}\label{H2}\ \ \  \fg^{h}=\fh\ \text{ and each non-zero eigenvalue
 of $\ad h$ has a non-zero real part.}\end{equation}

(The assumption on $\dim\fg$ ensures the existence of $h$).
We write $\fg=\fh\oplus (\oplus_{\alpha\in\Delta(\fg)}\fg_{\alpha})$
with  $\Delta(\fg)\subset\ft^*$ and
$$\begin{array}{l}

 \fg_{\alpha}:=\{g\in\fg|\ [h,g]=\alpha(h)g\ \text{ for all }h\in\ft\}.
 \end{array}$$
We introduce the triangular decomposition 
$\Delta(\fg)=\Delta^+(\fg)\coprod \Delta^-(\fg)$,
with
$$\Delta^{\pm}(\fg):=\{\alpha\in\Delta(\fg)|\ \pm Re\, \alpha(h)>0\},$$
 and  define the  partial order on $\ft^*$  by
$$\lambda>\nu\ \text{ if }\nu-\lambda\in\mathbb{N}\Delta^-.$$
We set 
$\fn^{\pm}:=\oplus_{\alpha\in\Delta^{\pm}}\fg_{\alpha}$ and consider 
the Borel subalgebra $\fb:=\fh\oplus\fn^+$.

\subsubsection{}\label{patpat}
Take  $z\in\ft$ satisfying
\begin{equation}\label{condz}
\alpha(z)\in\mathbb{R}_{\geq 0}\ \text{ for  }\alpha\in\Delta^+\ \text{ and }
\alpha(z)\in\mathbb{R}_{\leq 0}\ \text{ for }\alpha\in\Delta^-.
\end{equation}
Consider the superalgebra $\fp(z):=\fg^z+\fb$.
 Notice that 
$$\fp(z)=\fg^z\ltimes\fm,\ \text{ where }\ 
\fm(z):=\displaystyle\sum_{\alpha\in\Delta: \alpha(z)>0}\fg_{\alpha}.$$
Both triples $(\fp(z),\ft,h)$, $(\fg^z,\ft,h)$ satisfy (Ass$\ft$) and (\ref{H2}). One
has $(\fg^z)^{\ft}=\fp^{\ft}=\fh$ and
$$\begin{array}{l}
 \ \Delta^+(\fp(z))=\Delta^+(\fg),\ \ \ \ 
\ \Delta^+(\fg^z)=\{\alpha\in\Delta^{+}(\fg)|\ \alpha(z)=0\}\\
\Delta^-(\fp(z))=\Delta^-(\fg^z)=\{\alpha\in\Delta^{-}(\fg)|\ \alpha(z)=0\}
\end{array}$$

\subsubsection{Modules $M(\lambda),L(\lambda)$}\label{Mlambda}
For  a semisimple  $\ft$-module $N$ we
 denote by $N_{\nu}$ the weight space of the weight $\nu$.
 We denote by $\CO(\fg)$ the full subcategory of 
finitely generated modules with a diagonal action of $\ft$
and locally nilpotent action of $\fn$. Let $\Fin(\fg)$ be the  full subcategory 
of the finite-dimensional $\fg$-modules in $\CO(\fg)$.

By Dixmier generalization of Schur's Lemma (see~\cite{Dix}),
up to a parity change, the simple $\fh$-modules are 
parametrized by
$\lambda\in\ft^*$; we denote by $C_{\lambda}$ a simple
$\fh$-module, where $\ft$ acts by $\lambda$ (for each $\lambda$ we choose a grading on $C_{\lambda}$).
We view $C_{\lambda}$ as a $\fb$-module with the zero action of $\fn$
and set
$$M(\lambda):=\Ind^{\fg}_{\fb} C_{\lambda}.$$  
The module $M(\lambda)$
has a  unique simple quotient which we denote by 
$L(\lambda)$. We set
$$P^+(\fg):=\{\lambda\in\ft^*|\ \dim L(\lambda)<\infty\}.$$

We introduce similarly the modules
$M_{\fp}(\lambda)$, 
 $L_{\fp}(\lambda)$ for the algebra $\fp$.

\subsubsection{}\label{isotypical}
For $N\in\CO(\fg)$ we  will denote by $[N:L]$  the number of simple quotients isomorphic to $L$ or
to $\Pi L$ in a Jordan-H\"older  series of $N$. We say that
$N$ is an ``$L$-isotypical'' if all these quotients
are isomorphic to $L$ or
to $\Pi L$. 
 We introduce $\ext_{\fg}(\lambda;\nu)$ by~(\ref{extintro});
we will usually drop the index $\fg$ and write simply $\ext(\lambda;\nu)$.

\subsection{Remarks}
\label{remKM}
By~\Lem{lemExt} 
$\ \Ext^1_{\CO} (L(\lambda),L(\nu))=\Ext^1 (L(\lambda),L(\nu))$
if $\lambda\not=\nu$.

\subsubsection{}
If $\fg$ is a Kac-Moody superalgebra, then   $\{C_{\lambda}\}_{\lambda\in\ft^*}$
can be chosen in such a way that $\Ext^1_{\CO}  (L(\lambda),\Pi L(\nu))=0$ for all
$\lambda,\nu$ and thus $\ext(\lambda;\nu)=
\dim\Ext^1_{\CO} (L(\lambda),L(\nu))$.

If $\fg=\fq_n$, then $\ext(\lambda;\nu)\not=0$ implies that both
$C_{\lambda}, C_{\nu}$ are either  $\Pi$-invariant or not $\Pi$-invariant;
in this case for $\lambda\not=\nu$ one has
$\ext(\lambda;\nu)=\dim\Ext^1_{\CO} (L^{\Pi}(\lambda),L^{\Pi}(\nu))$, where
$L^{\Pi}(\lambda)$ is the ``$\Pi$-invariant simple module'' appeared in~\cite{PS1},\cite{PS2} 
(in other words, $L^{\Pi}(\lambda)$ is a simple $\fq_1\times\fq_n$-module).

\subsubsection{}\label{naoborot}
If $\fg$ is a Kac-Moody superalgebra, then $\fg$ admits antiautomorphism which stabilizes the elements of
$\ft$ and the category $\CO(\fg)$ admits
 a duality functor $\#$
with the property $L^{\#}\cong L$ for each simple module
$L\in\CO(\fg)$. 
By~\cite{Frisk}, for the $\fq_n$-case  $\CO(\fg)$ admits a duality functor $\#$ 
with the property $L^{\#}\cong L$ up to a parity change. In both cases
\begin{equation}\label{eqnaoborot}
\ext(\lambda;\nu)=\ext(\nu;\lambda).\end{equation}

\subsubsection{Set $\cN(\lambda;\nu;m)$}\label{cNm}
Let $\lambda\not=\nu\in\ft^*$ be such that  $Re (\lambda-\nu)(h)\geq 0$.
If 
$$0\to L(\nu)\to E\to L(\lambda)\to 0$$
is a non-split exact sequence, then $E$ is generated
by $E_{\lambda}\cong C_{\lambda}$, so
$E$ is a quotient of $M(\lambda)$ and $\nu<\lambda$.
For $\lambda,\nu\in\ft^*$ 
we denote by $\cN(\lambda;\nu;m)$ the set of  $\fg$-modules
$N$ satisfying
\begin{equation}\label{Nm}
\coSoc N\cong L(\lambda);\ \ \ \Soc N=\Rad N\ \text{ is $L(\nu)$-isotypical and }\ \ \ [N:L(\nu)]=m.
\end{equation}
 By~\Lem{lemExt} one has
$\ext(\lambda;\nu)=\max\{m|\ \ \cN(\lambda;\nu;m)\not=\emptyset\}$.
 Note that each module $N\in\cN(\lambda;\nu;m)$ is a quotient
of $M(\lambda)=\Ind^{\fg}_{\fb} L_{\fb}(\lambda)$.
We set
\begin{equation}\label{mgp}
m(\fg;\fp;\lambda;\nu):=\max\{m|\ \exists 
 N\in \cN(\lambda;\nu;m)\ \text{ which is a quotient of } \Ind^{\fg}_{\fp} L_{\fp}(\lambda)\}.
 \end{equation}

\subsubsection{}
\begin{cor}{corExtM}
Take $\lambda\not=\nu\in\ft^*$ with $Re (\lambda-\nu)(h)\geq 0$.
One has 
$$\ext(\lambda;\nu)=m(\fg;\fb;\lambda;\nu)\leq \dim M(\lambda)_{\nu}.$$
In particular, $\ext(\lambda;\nu)\not=0$ implies
$\lambda>\nu$.
\end{cor}

\subsection{Modules over $\fg\times\fh''$}\label{timesboxmod}
Let  $\fh''$ be a finite-dimensional Lie superalgebra satisfying
$[\fh''_{\ol{0}},\fh'']=0$.  Set
$$ \ft'':=\fh''_{\ol{0}},\ \ \ \fg'=\fg\times\fh'',\ \ 
\fh':=\fh\times\fh'',\ \ \ft'=\ft\times\ft'',\ \ 
\fp':=\fp\times\fh''.$$
Note that the triple $(\fg',\fh', h)$  
 satisfy the assumption (Ass$\ft$) and (\ref{H2}). 
 For $\lambda\in\ft^*$ and $\eta\in (\ft'')^*$
 denote by $\lambda\oplus \eta$ the corresponding element in $(\ft')^*$.
 Let $C_{\lambda}$, $C_{\eta}$, $C_{\lambda\oplus \eta}$
  be the corresponding $\fh$, $\fh''$ and $\fh'$-modules.

  \subsubsection{}\label{boxii}
 By~\cite{GSherman}, we can choose the grading on $C_{\lambda\oplus \eta}$ is such a way that
 $$C_{\lambda}\boxtimes C_{\eta}\cong\left\{\begin{array}{l}
C_{\lambda\oplus \eta}\oplus \Pi C_{\lambda\oplus \eta}\ \ \text{ if $C_{\lambda}$ and $C_{\eta}\ $ are $\Pi$-invariant}\\
  C_{\lambda\oplus \eta} \ \ \text{otherwise}.   \end{array}\right.
$$
Moreover if $C_{\eta}$ is not $\Pi$-invariant, then
 $C_{\lambda\oplus \eta}$ is $\Pi$-invariant if and only if
 $C_{\lambda}$ is $\Pi$-invariant.
 The similar statements hold for $M_{\fp}(\lambda\oplus \eta)$ and 
 for $L_{\fp}(\lambda\oplus \eta)$. 

 \subsubsection{}
If $\fh'=\ft'$, then $C_{\lambda},C_{\nu}, C_{\eta}$ are one-dimensional 
and 
$$\dim \Ext^1_{\CO}(L_{\fg}(\lambda),L_{\fg}(\nu))=\dim \Ext^1_{\CO}(L_{\fg'}(\lambda\oplus \eta),L_{\fg'}(\nu\oplus \eta)).$$ 
By~\ref{boxii} the same formula holds if
 $\lambda>\nu$ and  $C_{\eta}$ is not $\Pi$-invariant
(or if  $\lambda>\nu$ and both
  $C_{\nu}$, $C_{\lambda}$ are not $\Pi$-invariant). 
If $\fg=\fq_n$, then $\ext(\lambda;\nu)\not=0$ implies that both
$C_{\lambda}, C_{\nu}$ are either  $\Pi$-invariant or not $\Pi$-invariant
and the following corollary gives
$\ext(\lambda;\nu)=\ext(\lambda\oplus \eta;\nu\oplus \eta)$.

 \subsubsection{}
 \begin{cor}{corbox}
 \begin{enumerate}
 \item $\ext(\lambda\oplus \eta; \nu\oplus \eta')\not=0$ implies 
 $\eta'=\eta$ and $\ext(\lambda;\nu)\not=0$.
 \item
 Take $\lambda>\nu$. If at least one of the following conditions holds
 \begin{itemize}
 \item
  $C_{\eta}$ is not $\Pi$-invariant
  
  \item $C_{\nu}$ and $C_{\lambda}$ are not $\Pi$-invariant

 \item $C_{\nu}$ and $C_{\lambda}$ are  $\Pi$-invariant
  \end{itemize}
 then  
 the map $N\mapsto N\boxtimes C_{\eta}$ induces a bijection 
between the sets
 $\cN(\lambda;\nu;m)$ and  $\cN(\lambda\oplus\eta;\nu\oplus \eta;m)$. One has
$m(\fg;\fp;\lambda;\nu)=m(\fg';\fp';\lambda\oplus \eta;\nu\oplus \eta)$ and
$$\ext(\lambda;\nu)=\ext(\lambda\oplus \eta;\nu\oplus \eta).$$
  \end{enumerate}
 \end{cor}
 \begin{proof}
 Observe that if 
 $\eta'\not=\eta$, then the weight
$\lambda\oplus \eta-\nu\oplus \eta'$ does not lie in
 $\mathbb{Z}\Delta(\fg')$. Combining this observation with~\Cor{corExtM}
 we obtain $\ext(\lambda\oplus \eta; \nu\oplus \eta')=0$
 if $\eta\not=\eta'$. Other assertions follow from~\ref{boxii}. 
 \end{proof}

\subsection{}
The following lemma is a slight reformulation of Lemma 6.3 in~\cite{MS}.

\begin{lem}{lemExtVerma}
Take $\lambda,\nu\in\ft^*$ with $\lambda>\nu$. 

(i) $m(\fg;\fb;\lambda;\nu)\leq m(\fp;\fb;\lambda;\nu)$ if $\nu-\lambda\in \mathbb{N}\Delta^-(\fp)$;

(ii) 
$m(\fg;\fb;\lambda;\nu)=
m(\fg;\fp;\lambda;\nu)$  if $\nu-\lambda\not\in \mathbb{N}\Delta^-(\fp)$.
\end{lem}
\begin{proof}
For a semisimple $\ft$-module $N$ we denote by
 $\Omega(N)$ the set of weights of $N$.
 For each $\fg$-module $M$ we set
$$P_{\lambda}(M):=\{v\in M| \ zv=\lambda(z)v\}.$$
This defines an exact functor from $\fg-Mod$ to
$\fg^z-Mod$. Recall that  $\fp=\fg^z\ltimes\fm$ (see~\ref{patpat}).
Viewing
 $P_{\lambda}(M)$ as a $\fp$-module with the zero action of $\fm$
we obtain an exact functor 
$P_{\lambda}: \fg-Mod\to \fp-Mod$. By the PBW Theorem
$$
P_{\lambda}(M(\mu))=\left\{\begin{array}{ll}
M_{\fp}(\mu) &\ \text{ if } \mu(z)=\lambda(z)\\
0&\ \text{ if }  (\lambda-\mu)(z)>0\end{array}\right.
$$
Let us show that 
\begin{equation}\label{Res}
P_{\lambda}(L(\mu))=\left\{\begin{array}{ll}
L_{\fp}(\mu) &\ \text{ if }  \mu(z)=\lambda(z)\\
0&\ \text{ if } (\lambda-\mu)(z)>0.\end{array}\right.\end{equation}
Indeed, since $P_{\lambda}$ is exact, 
$P_{\lambda}(L(\mu))$ is a quotient of $P_{\lambda}(M(\mu))$; this gives the second formula.
For the first formula assume that $\mu(z)=\lambda(z)$ and that $E$ is a
proper submodule of 
$P_{\lambda}(L(\mu))$. 
Since $P_{\lambda}(L(\mu))$ is a quotient of $P_{\lambda}(M(\mu))$
and
$$(P_{\lambda}(M(\mu)))_{\mu}=(M_{\fp}(\mu))_{\mu}=C_{\mu}$$
is a simple $\fh$-module, one has $\gamma<\mu$ for each $\gamma\in\Omega(E)$. 
Since $E$ is a $\fp$-module and $\fg=\fn^-+\fp$, we have
$\cU(\fg)E=\cU(\fn^-)E$. Therefore $(\cU(\fn^-)E)_{\mu}=0$, so
$\cU(\fg)E$ is a proper $\fg$-submodule of $L(\mu)$. Hence $E=0$, so
 $P_{\lambda}(L(\mu))$ is simple. This  establishes~(\ref{Res}).

Fix $N\in \cN(\lambda;\nu;m)$ where $m:=m(\fg;\fb;\lambda;\nu)$.

For (i) consider the case when $\nu-\lambda\in \mathbb{N}\Delta^-(\fp)$.
Then $\lambda(z)=\nu(z)$, so~(\ref{Res}) gives
$$P_{\lambda}(L(\nu))=L_{\fp}(\nu),\ \ \ \ P_{\lambda}(L(\nu))=L_{\fp}(\lambda).$$
Since $P_{\lambda}$ is exact, 
$P_{\lambda}(N)$ is a quotient of $M_{\fp}(\lambda)$. Using~(\ref{Res}) we conclude 
that the $\fp$-module
$P_{\lambda}(N)$ lies in the set $\cN(\lambda;\nu;m)$ (defined for $\fp$ isntaed
of $\fg$).
Therefore $m\leq m(\fp;\fb;\lambda;\nu)$. This establishes (i).

For (ii)  assume that $\nu-\lambda\not\in \mathbb{N}\Delta^-(\fp)$.
 Let us show  that $N$ is a quotient of $\Ind_{\fp}^{\fg} L_{\fp}(\lambda)$. 
Write
$$\Ind_{\fp}^{\fg} L_{\fp}(\lambda)=M(\lambda)/J,\ \ \ L_{\fp}(\lambda)=M_{\fp}(\lambda)/J'$$
where $J$ (resp., $J'$) is the corresponding submodule of $M(\lambda)$
(resp., of  $M_{\fp}(\lambda)$).
Since $\Ind^{\fg}_{\fp}$ is exact and $\Ind^{\fg}_{\fp} M_{\fp}(\lambda)=M(\lambda)$
one has
$J\cong \Ind^{\fg}_{\fp} J'$; in particular, 
 each maximal element in $\Omega(J)$ lies in $\Omega(J')$.
Note that 
$$\Omega(J')\subset \lambda+\mathbb{N}\Delta^-(\fp).$$

Let $\phi: M(\lambda)\twoheadrightarrow N$ be the canonical surjection.
Since $J_{\lambda}=0$, $\phi(J)$ is a proper submodule
of $N$, so  $\phi(J)$ is a submodule of
$\Soc(N)=L(\nu)^{\oplus m}$. 

If $\phi(J)\not=0$, then $\nu$ is a maximal element
in $\Omega(J)$ and so $\nu\in \lambda+\mathbb{N}\Delta^-(\fp)$
which contradicts to $\nu-\lambda\not\in\mathbb{N}\Delta^-(\fp)$.
Therefore $\phi(J)=0$, so $\phi$ induces a map 
$$\Ind_{\fp}^{\fg} L_{\fp}(\lambda)=M(\lambda)/J \twoheadrightarrow N.$$
Hence $N$ is a quotient of $\Ind_{\fp}^{\fg} L_{\fp}(\lambda)$
which gives $m\leq  m(\fg;\fp;\lambda;\nu)$.
Since $\Ind_{\fp}^{\fg} L_{\fp}(\lambda)$ is a quotient of $M(\lambda)$,
we have 
$m(\fg;\fb;\lambda;\nu)\geq m(\fg;\fp;\lambda;\nu)$.
Thus $m(\fg;\fb;\lambda;\nu)=m(\fg;\fp;\lambda;\nu)$ as required.
\end{proof}

\subsection{}\label{zi}
Take $z_1,\ldots,z_{k-1}\in\ft$ satisfying~(\ref{condz}) and 
the condition $\fg^{z_i}\subset\fg^{z_{i+1}}$. Setting 
$\fg^{(i)}:=\fg^{z_i}$ we obtain the chain
\begin{equation}\label{chain}
\fh=:\fg^{(0)}\subset\fg^{(1)}\subset \fg^{(2)}\subset\ldots\subset \fg^{(k)}:=\fg.\end{equation}
We introduce
$\tilde{\fp}^{(i)}:=\fg^{(i)}+\fb$ and 
$\fp^{(i)}:=\tilde{\fp}^{(i-1)} \cap \fg^{(i)}$ with $\tilde{\fp}^{(0)}=\fp^{(0)}:=\fb$;
note that $\fp^{(i)}$ (resp., $\tilde{\fp}^{(i)}$) is a parabolic subalgebra
 in $\fg^{(i)}$ (resp., in $\fg$).

\subsubsection{}\label{formulaeforpt}
Taking $z_0:=h$ as in (\ref{H2}) and $z_k:=0$ we obtain for $s=1,\ldots,k$
$$\begin{array}{ll}
\tilde{\fp}^{(s-1)}=\displaystyle\sum_{\alpha:\ \alpha(z_{s-1})\geq 0}\fg_{\alpha}\ \subset\ 
\tilde{\fp}^{(s)}=\displaystyle\sum_{\alpha:\ \alpha(z_{s})\geq 0}\fg_{\alpha},\\ 
{\fp}^{(s)}=\displaystyle\sum_{\substack{\alpha:\ \alpha(z_{s-1})\geq 0\\
\alpha(z_s)=0}}\fg_{\alpha}.
\end{array}$$

In particular, $\tilde{\fp}^{(s)}=\fg^{(s)}\ltimes\tilde{\fm}^{(s)}$
and ${\fp}^{(s)}=\fg^{(s-1)}\ltimes {\fm}^{(s)}$
where
$$ \tilde{\fm}^{(s)}:=\sum_{\alpha: \alpha(z_s)>0} \fg_{\alpha},\ \ \ \ \ \ \ \ \fm^{(s)}:=\displaystyle\sum_{\substack{\alpha:\ \alpha(z_{s-1})>0\\
\alpha(z_s)=0}}\fg_{\alpha}.$$
One has 
$\tilde{\fm}^{(i+1)}\subset\tilde{\fm}^{(i)}$
(since  $\fg^{(i)}\cap \fn$ can be identified with $\fn/\tilde{\fm}^{(i)}$).

\subsubsection{}
\begin{cor}{corExtchain}
For $\lambda>\nu$ one has
$$\ext(\lambda;\nu)\leq m(\fg^{(s)};\fp^{(s)};\lambda;\nu)=
\ext_{\fg^{(s)}}(\lambda;\nu)$$
where $s$ is minimal such that $\nu-\lambda\in\mathbb{N}\Delta^-(\fg^{(s)})$.
\end{cor}
\begin{proof} 
Combining~\Cor{corExtM} and~\Lem{lemExtVerma} we obtain
$$\ext(\lambda;\nu)=m(\fg;\fb;\lambda;\nu)\leq m(\tilde{\fp}^{(s)};\tilde{\fp}^{(s-1)};\lambda;\nu).$$
The $\tilde{\fp}^{(s)}$-module
 $N:=\Ind^{\tilde{\fp}^{(s)}}_{\tilde{\fp}^{(s-1)}} L_{\tilde{\fp}^{(s-1)}}(\lambda)$ is generated by its highest weight space $N_{\lambda}$,
Since $\tilde{\fm}^{(s)}\subset\fn$ is an ideal of $\fp^{(s)}$, $N^{\tilde{\fm}^{(s)}}$
is a submodule of $N$, which implies
$\tilde{\fm}^{(s)} N=0$. Hence 
 $N$ is a module over $\tilde{\fp}^{(s)}/\tilde{\fm}^{(s)}=\fg^{(s)}$.
This gives
$$m(\tilde{\fp}^{(s)};\tilde{\fp}^{(s-1)};\lambda;\nu)=m(\fg^{(s)};\tilde{\fp}^{(s-1)}/\tilde{\fm}^{(s)};\lambda;\nu).$$
Using the formulae from~\ref{formulaeforpt} we see
that  $\fp^{(s)}$ is the image of $\tilde{\fp}^{(s-1)}$ in $\tilde{\fp}^{(s)}/\tilde{\fm}^{(s)}=\fg^{(s)}$.  Therefore
$m(\fg^{(s)};\tilde{\fp}^{(s-1)}/\fm^{(s)};\lambda;\nu)=m(\fg^{(s)};\fp^{(s)};\lambda;\nu)$
and thus
\begin{equation}\label{ext11}
\ext(\lambda;\nu)\leq m(\fg^{(s)};\fp^{(s)};\lambda;\nu).
\end{equation}

Clearly, $m(\fg^{(s)};\fp^{(s)};\lambda;\nu)\leq \ext_{\fg^{(s)}}(\lambda;\nu)$. Using~(\ref{ext11}) for
$\fg^{(s)}$ we obtain
$$\ext_{\fg^{(s)}}(\lambda;\nu)\leq m(\fg^{(s)};\fp^{(s)};\lambda;\nu).$$
Thus $\ext_{\fg^{(s)}}(\lambda;\nu)= m(\fg^{(s)};\fp^{(s)};\lambda;\nu)$. Now~(\ref{ext11}) gives the required formula.
\end{proof}

%
%
%

\subsection{Functors $\Gamma_{\bullet}^{\fg,\fp}$}\label{Gamma}

For a  parabolic subalgebra
$\fp\subset \fg$ and a finite-dimensional $\fp$-module $V$  we denote by $\Gamma^{\fg,\fp}(V)$
a maximal finite-dimensional quotient of $\Ind^{\fg}_{\fp} (V)$.
It is easy to see that this quotient is unique and
that for any  finite-dimensional quotient $N$ of  $\Ind^{\fg}_{\fp} (V)$
there exists an epimorphism $\Gamma^{\fg,\fp}(V) \twoheadrightarrow N$.

\subsubsection{}\label{Gammagp}
In~\cite{Penkov}, I.~Penkov introduced important functors from $\Fin(\fp)$ to $\Fin(\fg)$. We will use a modification of these functors which appeared in~\cite{GS} and other papers. These functors 
$\Gamma_{\bullet}=\{\Gamma_i\}_0^{\infty}$ have the following properties
\begin{itemize}
\item
$\Gamma_0^{\fg,\fp}(V)=\Gamma^{\fg,\fp}(V)$;
\item Each short exact sequence of $\fp$-modules
$$0\to U\to V\to U'\to 0$$
induces a long exact sequence
$$\ldots\to \Gamma_1^{\fg,\fp}(V)\to \Gamma_1^{\fg,\fp}(U')\to \Gamma_0^{\fg,\fp}(U)\to \Gamma_0^{\fg,\fp}(V)\to
\Gamma_0^{\fg,\fp}(U')\to 0.$$
\end{itemize}

Until the end of this section we assume the existence of  $\Gamma_{\bullet}$ satisfying the above properties. Observe that  $[\Gamma^{\fg,\fp}_0(L_{\fp}(\lambda)):
 L_{\fg}(\lambda)]=1$ if $\lambda\not\in P^+(\fg)$: we set
$$\begin{array}{l}
K^j(\lambda;\nu):=[\Gamma_j^{\fg,\fp}(L_{\fp}(\lambda)):
 L_{\fg}(\nu)]-\delta_{0j}\delta_{\lambda\nu}.\end{array}$$
 Observe that
$\Gamma^{\fg,\fp}(L_{\fp}(\lambda))=0$
if $\lambda\not\in P^+(\fg)$; for $\lambda\in P^+(\fg)$ one has
 \begin{equation}\label{RadN}
\coSoc \Gamma^{\fg,\fp}(L_{\fp}(\lambda))=L_{\fg}(\lambda)\ \ \ \ K^0(\lambda;\mu)=[\Rad (\Gamma^{\fg,\fp}(L_{\fp}(\lambda))): L_{\fg}(\mu)].
\end{equation}
 In particular,
 $K^0(\lambda;\nu)\not=0$ implies $\nu<\lambda$.

\subsubsection{}
\begin{lem}{lem331}
Let $\lambda,\nu\in P^+(\fg)$ with $\nu<\lambda$ be such that
$$\begin{array}{llll}
\forall\mu\not=\nu\ & \ K^0(\lambda;\mu)\not=0 & \Longrightarrow\ &
\ext(\mu;\nu)=0 \end{array}.$$
If  $K^0(\lambda;\nu)=1$ or $\ext(\nu;\nu)=0$,
then $K^0(\lambda;\nu)\leq \ext(\lambda;\nu)$.
\end{lem}
\begin{proof}
By~(\ref{RadN}) in both cases the isotypical component of
$L_{\fg}(\nu)$ is a direct summand of $\Rad \Gamma^{\fg,\fp} (L_{\fp}(\lambda))$, so~\Lem{lemExt} (i) gives the required  inequality.
\end{proof}

\subsubsection{}
Retain notation of~\ref{patpat} and recall that $\fp=\fg^z\ltimes \fm$. 

\begin{lem}{lem332}
Let $\lambda,\nu\in P^+(\fg)$ with $\nu<\lambda$ be such that
$$\begin{array}{lllll}
(a) & & K^1(\lambda;\nu)=0\\
(b) & \forall\mu & K^0(\lambda;\mu)\not=0 & \Longrightarrow\ &
\ext(\nu;\mu)=0.\end{array}$$

Then $\ext_{\fg^z}(\lambda;\nu) \leq \ext (\lambda;\nu).$
\end{lem}
\begin{proof}
Without loss of generality we will assume that $m:=\ext_{\fg^z}(\lambda;\nu)>0$.
 By~\Lem{lemExt}
there exists an indecomposable $\fg^{z}$-module $N_1$ with a short exact sequence
$$0\to L_{\fg^{z}}(\nu)^{\oplus m}\to N_1\to L_{\fg^{z}}(\lambda)\to 0$$ 
where $L_{\fg^{z}}(\nu)^{\oplus m}$ stands for the direct sum of $m_0$ copies
of $L_{\fg^{z}}(\nu)$ and $m_1$ copies
of $\Pi L_{\fg^{z}}(\nu)$ with $m_0+m_1=m$.
Since  $\fp=\fg^{z}\ltimes\fm$, the corresponding $\fp$-module 
$N_2:=\Res_{\fp}^{\fg^{z}} N_1$ 
is an indecomposable module with a short exact sequence
$$0\to L_{\fp}(\nu)^{\oplus m}\to N_2 \to L_{\fp}(\lambda)\to 0.$$
Consider the corresponding  long exact sequence of $\fg$-modules
$$\begin{array}{l}
\ldots \to \Gamma_1^{\fg,\fp} (L_{\fp}(\lambda))^{\oplus m}\overset{\phi}{\to}
\Gamma_0^{\fg,\fp}
(L_{\fp}(\nu))^{\oplus m}
\to\Gamma_0^{\fg,\fp}(N_2)\to \Gamma_0^{\fg,\fp}(L_{\fp}(\lambda))\to 0.\end{array}$$
Recall that $\ \coSoc \Gamma_0^{\fg,\fp}
(L_{\fp}(\nu))=L_{\fg}(\nu)$.
Since $K^1(\lambda;\nu)=0$ the image of $\phi$ lies in $\Rad \Gamma_0^{\fg,\fp}
(L_{\fp}(\nu))^{\oplus m}$. Thus $\Gamma_0^{\fg^{},\fp^{}}(N_2)$ has a quotient $N_3$  with
the short exact sequence
\begin{equation}\label{N3}
0\to L_{\fg^{}}(\nu)^{\oplus m}\to N_3\to \Gamma_0^{\fg^{},\fp^{}}(L_{\fp^{}}(\lambda))\to 0.\end{equation}
Since $N_2$ is indecomposable, it is generated by 
its $\lambda$-weight space $(N_2)_{\lambda}$. Since
$N_3$ is a  quotient of $\Gamma_0^{\fg^{},\fp^{}}(N_2)$
which is a quotient of $\Ind^{\fg^{}}_{\fp^{}}(N_2)$,
$N_3$ is also generated by 
its $\lambda$-weight space. Hence $N_3$ is indecomposable and 
$$\coSoc (N_3)\cong L_{\fg}(\lambda)\cong \coSoc\,\bigl(
\Gamma_0^{\fg^{},\fp^{}}(L_{\fp^{}}(\lambda))\bigr).$$
The short exact sequence~(\ref{N3}) induces 
 a  short exact sequence
$$0\to L_{\fg^{}}(\nu)^{\oplus m}\to \Rad(N_3)\to M\to 0,$$
where $M:=
\Rad\bigl(\Gamma_0^{\fg^{},\fp^{}}(L_{\fp^{}}(\lambda))\bigr)$.
This sequence splits since,  the assumption (b) gives $\ext(\nu;\mu)=0$ if $L_{\fg^{}}(\mu)$ is a subquotient of $M$. Hence $M$ is a submodule of $N_3$,
which gives the following  short exact sequence
$$0\to L_{\fg^{}}(\nu)^{\oplus m}\to N_3/M\to  L_{\fg}(\lambda)\to 0.$$ 
By above, $N_3/M$ is generated by its $\lambda$-weight space, so
$N_3/M$ is indecomposable. \Lem{lemExt} (i) gives
$\ext(\lambda;\nu)\geq m$
as required.
\end{proof}

\subsection{}\label{zii}
Retain notation and assumption of~\ref{zi}. 
 In all formulae  where ${(s)}$ appears, $s$ is assumed to be one of the numbers $1,\ldots,k$. For each $s$ we 
 fix a decomposition
$\fg^{z_s}=\fg_{(s)}\times \fh^{\perp}_{(s)}$
in such a way that $\fh^{\perp}_{(s)}\subset \fh$ and
$$\fg_{(0)}\subset \fg_{(1)}\subset \ldots\subset \fg_{(k)}.$$
We set 
$\fh_{(s)}:=\fg_{(s)}\cap\fh$, $\ft_{(s)}:=\fg_{(s)}\cap\ft$, $
\ft^{\perp}_{(s)}:=\fh^{\perp}_{(s)}\cap \ft$ and 
${\fp}_{(s)}:=(\fg_{(s-1)}+\fb)\cap \fg_{(s)}$.
Note that $\fp_{(s)}$ is s a parabolic subalgebra in $\fg_{(s)}$; one has
$$\fg^{z_s}=\fg_{(s)}+\fh,\ \ \ \fh=\fh_{(s)}\times \fh^{\perp}_{(s)},\ \ \fp^{(s)}=\fp_{(s)}\times \fh^{\perp}_{(s)},\ \ \ {\fp}_{(s)}=\fp^{(s)}\cap \fg_{(s)}.$$ 
In   the notation of~\ref{timesboxmod}
we have $\ P^+(\fg^{z_s})=P^+(\fg_{(s)})\oplus (\ft_{(s)}^{\perp})^*$.
Observe that
$$\ft^*=P^+(\fg^{z_0})\supset P^+(\fg^{z_1})\supset P^+(\fg^{z_2})\supset\ldots\supset P^+(\fg^{z_k})=P^+(\fg).$$

\subsubsection{}\label{AiB}
We assume that for each $s$  one has
\begin{itemize}
\item[(A)]
for any $\lambda',\nu'\in P^+(\fg_{(s)})$ with $\ext_{\fg_{(s)}}(\lambda',\nu')\not=0$  the simple $\fh_{(s)}$-modules
$C_{\lambda'},C_{\nu'}$ are either $\Pi$-invariant
or not  $\Pi$-invariant simultaneously;

\item[(B)] there exists
$\Gamma_{\bullet}^{\fg_{(s)},{\fp}_{(s)}}:\ 
\Fin({\fp}_{(s)})\to\Fin(\fg_{(s)})$ satisfying
the conditions~\ref{Gammagp}.
\end{itemize}

Observe that (A) holds if $\fh_{\ol{1}}=0$ (that is $\fh=\ft$);
in addition (A) holds 
if $\fg_{(s)}\cong \fq_m$ (this
follows from the description of the centre of $\cU(\fq_m))$
obtained in~\cite{Sq}).

\subsubsection{}\label{defKjs}
Take $\lambda,\nu\in \ft^*$  and set $\lambda':=\lambda|_{\ft_{(s)}}$, $\nu':=\nu|_{\ft_{(s)}}$. We introduce
\begin{equation}\label{defKji}
K^j_{(s)}(\lambda;\nu):=\left\{
\begin{array}{ll}
0 &  \text{ if } \lambda\not\in P^+(\fg_{(s)})\\
0 &  \text{ if } \lambda|_{\ft_{(s)}^{\perp}}\not=\nu|_{\ft_{(s)}^{\perp}}\\
\ [\Gamma_j^{\fg_{(s)},\fp_{(s)}}(L_{\fp_{(s)}}(\lambda')):
 L_{\fg_{(s)}}(\nu')]\ & \text{ if } \lambda|_{\ft_{(s)}^{\perp}}=\nu|_{\ft_{(s)}^{\perp}}.
\end{array}
\right.\end{equation}

Note that 
\begin{equation}\label{K0i}
K^0_{(s)}(\lambda;\nu)\not=0 \ \ \Longrightarrow\ \ 
\nu\in\lambda+\mathbb{N}\Delta^-(\fg_{(s)}),\ \ \nu|_{\ft_{(s)}}\in P^+(\fg_{(s)}).\end{equation}

We set
$$ \ext_{(s)}(\lambda;\nu):= \ext_{\fg^{z_s}}(\lambda;\nu).$$

Combining the assumption (A) and~\Cor{corbox}  we get for $\nu<\lambda$
\begin{equation}\label{extzi}
 \ext_{(s)}(\lambda;\nu)=\left\{
\begin{array}{lcl}
0 & & \text{ if } \lambda|_{\ft_{(s)}^{\perp}}\not=\nu|_{\ft_{(s)}^{\perp}}\\
\ext_{\fg_{(s)}}(\lambda';\nu') & & \text{ if } \lambda|_{\ft_{(s)}^{\perp}}=\nu|_{\ft_{(s)}^{\perp}}
\end{array}
\right.\end{equation}
Note that $m(\fg^{(s)};\fp^{(s)};\lambda;\nu)\leq K^0_{(s)}(\lambda;\nu)$.

\subsection{Graph $G(\ft^*;K^0)$}\label{graphtK0}
For $\lambda,\nu\in\ft^*$ we introduce
$$
s(\lambda;\nu):=\max\{s|\ \lambda|_{\ft_{(s)}^{\perp}}=\nu|_{\ft_{(s)}^{\perp}}
\},\ \ \ \ \ k_0(\lambda;\nu):=K^0_{(s(\lambda;\nu))}(\lambda;\nu),\ \ \ $$
Note that 
$s(\lambda;\nu)=\min\{s|\ \nu-\lambda\in\mathbb{N}\Delta^-(\fg_{z_s})\}$
if $\nu<\lambda$. \Cor{corExtchain}  gives
\begin{equation}\label{dimExtK0}
  \ext(\lambda;\nu)\leq  \ext_{(s(\lambda;\nu))}(\lambda;\nu)\leq k_0(\lambda;\nu)\ \text{ for each }\ 
  \lambda,\nu\in P^+(\fg) \text{ with } \nu<\lambda.
\end{equation}

\subsubsection{Definitions}\label{graphs12}
We say that $(\lambda;\nu)$ is {\em $K^i$-stable} if 
$K^i_{(s)}(\lambda;\nu)\not=0$ for each $s>s(\lambda;\nu)$.

Let $G(\ft^*;K^0)$ be the graph with the set of vertices 
$\ft^*$ connected by $k_0(\lambda;\nu)$-edges of the form $\nu\to\lambda$.

For each $B\subset \ft^*$ we denote by $G(B,K^0)$ 
the induced subgraph of $G(\ft^*;K^0)$.
We say that a graph $G(B,K^0)$ is {\em bipartite} if there exists
$\pari: B\to \mathbb{Z}_2$ such that $\nu\to\lambda$ implies
$\pari(\nu)\not=\pari(\lambda)$. For each $\lambda\in \ft^*$ let $B(\lambda)$ 
be the set consisting of $\lambda$ and all its direct predessors in
$G(\ft^*;K^0)$, i.e.
$$B(\lambda):=\{\lambda\}\cup\{\nu|\ k_0(\lambda;\nu)\not=0\}.$$

\subsubsection{Remarks}
Observe that $G(\ft^*;K^0)$ is a directed graph without cycles
(for any edge $\mu\to\nu$ one has $\mu<\nu$).
For $\fg\not=\fgl(n|n)$ one has $B(0)=\{0\}$
since $0$ is a minimal weight in $P^+(\fg_{(i)})$
for each $i$.

Note that if $G(B(\lambda);K^0)$ is bipartite and $\ext(\nu;\nu)=0$ for each
$\nu\in B(\lambda)$, then the radical of  $\Gamma^{\fg,\fp} L_{\fp}(\lambda)$
is semisimple.

\subsubsection{}
\begin{cor}{corgraphs}
Let $\lambda\in P^+(\fg)$ be such that
\begin{itemize}
\item[(a)] $B(\lambda)\subset P^+(\fg)$ and $G(B(\lambda);K^0)$ is bipartite;
\item[(b)] $\ext_{(s)}(\mu;\nu)=0\ \ \Longleftrightarrow\ \ext_{(s)}(\nu;\mu)=0$ for all $s$ and $\mu,\nu\in B(\lambda)\setminus\{\lambda\}$.
\end{itemize}

\begin{enumerate}
\item
If $\nu\in  P^+(\fg)$ with $\nu<\lambda$ satisfies
\begin{itemize}
\item[(c)]
$(\lambda;\nu)$ is $K^1$-stable;
\item[(d)] $(\lambda;\nu)$ is $K^0$-stable or $\ext_{(s)}(\nu;\nu)=0$ for each $s$,
\end{itemize}
then $\ext(\lambda;\nu)=\ext_{(s_0(\lambda;\nu))}(\lambda;\nu)$.
\item If $\lambda,\nu$ satisfy (a)--(d) and
\begin{itemize}
\item[(e)]
$k_0(\lambda;\nu)=1$ or
 $\ext_{(s(\lambda;\nu))}(\nu;\nu)=0$.
\end{itemize}
then $\ext(\lambda;\nu)=k_0(\lambda;\nu)$.
\end{enumerate}
\end{cor}
\begin{proof}  
If $\nu\not\in B(\lambda)$, then~(\ref{dimExtK0}) gives
$\ext(\lambda;\nu)=k_0(\lambda;\nu)=0$. Assume that $\nu\in B(\lambda)$.
Set
$$p:=s_0(\lambda;\nu).$$
Combining~(\ref{extzi}), (\ref{dimExtK0}) we obtain $\ext_{\fg^{z_s}}(\mu_2;\mu_1)\leq k_0(\mu_2;\mu_1)$ for  $\mu_1,\mu_2\in P^+(\fg)$ if $\mu_2>\mu_1$.
Then the assumptions (a), (b) give 
\begin{equation}\label{exter}
\ext_{(s)}(\mu_1;\mu_2)=0 \ \text{ for all }\ \mu_1\not=\mu_2\in B(\lambda)\setminus\{\lambda\}  \text{ and each }s.\end{equation}
  Take $s>p$
and view $\lambda,\nu$  as elements of $P^+(\fg^{z_s})$.
We will use~\Lem{lem332} for the pair $\fp^{(s)}\subset \fg^{z_s}$.
Let us check the assumptions of this lemma:
the assumption (a)  follows from (c) and the assumption (b)
follows from~(\ref{exter}) for $\mu\not=\nu$ (since $\nu\in B(\lambda)$);
the assumption (b) for $\mu=\nu$ means that $K^0_{(s)}(\lambda;\nu)=0$
implies $\ext_{(s)}(\nu;\nu)=0$--- this follows from (d). 
\Lem{lem332} gives $\ext_{(s)}(\lambda;\nu)\leq \ext_{(s+1)}(\lambda;\nu)$.
Using~(\ref{dimExtK0}) we get 
$$\ext(\lambda;\nu)\leq \ext_{(p)}(\lambda;\nu)\leq \ext_{(n)}(\lambda;\nu)=
\ext(\lambda;\nu).$$
This proves (i). For (ii) note that~(\ref{exter}) and (e)
imply the assumptions of~\Lem{lem331} for $\fg^{z_p}$  
which gives $K^0_{p}(\lambda;\nu)\leq 
\ext_{(p)}(\lambda;\nu)$. By (i) this can be rewritten as  
 $k_0(\lambda;\nu)\leq \ext(\lambda;\nu)$. Now~(\ref{dimExtK0}) 
gives  $k_0(\lambda;\nu)=\ext(\lambda;\nu)$ as required.
\end{proof}

\subsubsection{Remark}
If $\lambda$ satisfies (a), (b) and the assumption (e) 
holds for each $\nu\in B(\lambda)$,
then $\Gamma^{\fg,\fp} L_{\fp}(\lambda)$ has a semisimple radical.

\subsubsection{Remark}
In the examples considered  below   each pair $(\lambda;\nu)$ with $\lambda\not=\nu$ 
is $K^i$-stable for any $i$
(in fact $K_{(s)}^i(\lambda;\nu)\not=0$ implies $K_{(s')}^{i'}(\lambda;\nu)\not=0$
for each $s'\not=s$ and any $i'$). In most of the cases $G(B(\lambda);K^0)$ is bipartite (this simply means that $k_0(\mu_1;\mu_2)=0$ for $\mu_1,\mu_2\in B(\lambda)\setminus \{\lambda\}$); moreover, 
$K_{(s)}^i(\lambda;\nu)\not=0$ implies $\pari(\nu)\equiv \pari(\lambda)+i$ modulo $2$.

\section{Weight diagrams and arch diagrams}\label{arches}
In this section we introduce the language of ``arch diagrams''
which will be used in Section~\ref{Section3}. We will consider the following examples
\begin{itemize}
\item
 the principal block over
 $\fg=\fgl(n|n),\osp(2n+t|2n)$ for $t=0,1,2$;
 \item  the principal block over
 $\fq_{2n+\ell}$ for $\ell=0,1$;
 \item  the  ``half-integral'' block of maximal atypicality over $\fq_{2n}$.
 \end{itemize}
We set $\ell:=\dim\ft-2n$, 
i.e., 
$\ell=1$ for $\osp(2n+2|2n)$, $\fq_{2n+1}$ and $\ell=0$ in other cases.

 \subsection{Triangular decompositions}\label{tria}
We fix the following bases of simple roots
$$\Sigma:=\left\{\begin{array}{ll}
\vareps_1-\vareps_2,\ldots,\vareps_n-\delta_1,\delta_1-\delta_2,\ldots,\delta_{n-1}-\delta_n\ & \text{ for }\fgl(n|n)\\
\vareps_1-\delta_{1},\delta_1-\vareps_2,\ldots,\vareps_{n}-\delta_{n},
\delta_n\ & \text{ for }\osp(2n+1|2n)\\
\delta_{1}-\vareps_1,\vareps_1-\delta_{2},\ldots,\vareps_{n-1}-\delta_{n},
\delta_n\pm\vareps_n & \text{ for } \osp(2n|2n)\\
\vareps_{1}-\delta_1,\delta_1-\vareps_{2},\ldots,\vareps_{n}-\delta_n,\delta_n\pm\vareps_{n+1} & \text{ for }
\osp(2n+2|2n).\\
\vareps_1-\vareps_2,\ldots,\vareps_{2n+\ell-1}-\vareps_{2n+\ell} \ & \text{ for } \fq_{2n+\ell}
\end{array}\right.$$
and take the following Weyl vector 
$$\rho:=\left\{\begin{array}{ll}
\displaystyle\sum_{i=1}^n (n-i)(\vareps_{i}-\delta_{n+1-i})\ \ \  & 
\text{ for }\fgl(n|n)\\
0  &  \text{ for }
\osp(2n|2n), \osp(2n+2|2n), \fq_{2n+\ell}\\
\displaystyle\frac{1}{2}\sum_{i=1}^n(\delta_i-\vareps_i)& \text{ for }\osp(2n+1|2n).\\
\end{array}\right.$$

\subsection{Weight diagrams}\label{lambda+rho}
We denote by $\cB_0$ the set of the highest weights for simple modules 
lying in the principal block of $\Fin(\fg)$; for $\fq_{2n}$ we denote by  
$\cB_{1/2}$ the set of the highest weights for simple modules 
lying in the half-integral  block of maximal atypicality.
In what follows $\cB$  will denote $\cB_0$ or $\cB_{1/2}$. These sets 
can be described as follows.

\begin{itemize}
\item[$\bullet$] For $\fgl(n|n)$ the set $\cB_0$ consists of $\lambda$s such that 
$\lambda+\rho=\displaystyle\sum_{i=1}^{n}\lambda_i(\vareps_i-\delta_i)$, where
$\lambda_1,\ldots,\lambda_n$
 are  integers with $\lambda_{i+1}<\lambda_i$. 

\item[$\bullet$] For $\osp(2n+t|2n)$ the set
$\cB_0$ consists of $\lambda$s such that 
$$\lambda+\rho=\left\{
\begin{array}{ll}
\displaystyle\sum_{i=1}^{n-1} \lambda_i(\vareps_i+\delta_i)+\lambda_n(\delta_n+\xi \vareps_n) &\text{ for }  t=0\\
\ \ \ \ \ \displaystyle\sum_{i=1}^{n} \lambda_i(\vareps_i+\delta_i)&\text{ for }  t=2\\
\displaystyle\sum_{i=1}^{s-1} (\lambda_i+\frac{1}{2})
(\vareps_i+\delta_i)+\frac{1}{2}(\delta_s+\xi  \vareps_s)+\displaystyle\sum_{i=s+1}^{n} \frac{1}{2}
(\delta_i-\vareps_i)&\text{ for }  t=1\end{array}
\right.$$
where $\xi\in\{\pm 1\}$ and $\lambda_1,\ldots,\lambda_n\in\mathbb{N}$
 with $\lambda_{i+1}<\lambda_i$ or $\lambda_i=\lambda_{i+1}=0$.
For $t=1$ we have $1\leq s\leq n+1$ and we set
$\lambda_s:=\lambda_{s+1}:=\ldots=\lambda_n=0$ if  $s\leq n$ (for $s=n+1$ we have
 $\lambda+\rho=\sum_{i=1}^{n} (\lambda_i+\frac{1}{2})
(\vareps_i+\delta_i)$).

\item[$\bullet$] 
For $\fq_{2n+\ell}$ the set  $\cB_0$ consists of $\lambda$s such that 
$$\lambda+\rho=\displaystyle\sum_{i=1}^{n}\lambda_i(\vareps_i-\vareps_{2n+\ell+1-i}),$$ where
$\lambda_1,\ldots,\lambda_n\in\mathbb{N}$ with $\lambda_{i+1}<\lambda_i$ or $\lambda_{i+1}=\lambda_i=0$.

\item[$\bullet$] 
For $\fq_{2n}$ the set  $\cB_{1/2}$ consists of $\lambda$s such that 
$$\lambda+\rho=\displaystyle\sum_{i=1}^{n}\lambda_i(\vareps_i-\vareps_{2n+1-i}),$$ where
$\lambda_1,\ldots,\lambda_n\in\mathbb{N}+1/2$ 
and $\lambda_{i+1}<\lambda_i$.
\end{itemize}

\subsubsection{}\label{wtdiag}
We assign to $\lambda$  as above a ``weight diagram'':
for $\cB_0$ (resp., $\cB_{1/2}$) the weight diagram is
a number line with one or several symbols drawn at each position with  integral (resp., half-integral) coordinate:

\begin{itemize}
\item
we put the sign $\times$ at each  position with the coordinate $\lambda_i$;
\item
if $\ell=1$  we add $>$ at the zero position;

\item we add the ``empty symbol'' $\circ$ to all empty positions;

\item for   $\osp(2n|2n)$
 with $\lambda_k\not=0$ 
and for $\osp(2n+1|2n)$ with $s\leq k$, we write the 
sign of $\xi$ before the diagram  ($+$ if $\xi=1$ and $-$ if $\xi=-1$).
\end{itemize}
Note that $\lambda\in\cB_0$ (resp., $\lambda\in\cB_{1/2}$) is  uniquely determined by the weight diagram constructed by the above procedure.

For a  diagram $f$ 
we denote by $f(a)$ the symbols at the position $a$ (for $\fgl(n|n)$ one has
$f(a)\in\{\circ,\times\}$).
For  $\osp(2n|2n)$ (resp., $\osp(2n+1|2n)$) a  diagram  has a sign
if and only if $f(0)=\circ$ (resp.,  $f(0)\not=\circ$). We say that two weight diagrams  ``have different signs'' if
one of them has sign $+$ and another sign $-$.

\subsubsection{}
Consider the case $\fg\not=\fgl(n|n)$.
In this case each position with negative coordinate contains $\circ$ and
we will not depcit these positions. 
Each position with a positive coordinate contains either $\times$ or $\circ$.
For $\ell=0$  the zero position is occupied either by $\circ$ or by
several symbols $\times$; we
write this as $\times^i$ for $i\geq 0$. Similarly, for $\ell=1$
  the zero position is occupied by $\overset{\times^i}{>}$ with $i\geq 0$.
%
%

\subsubsection{Examples}
For $\fgl(3|3)$ the weight diagram of $0$ is 
$\ldots\circ\circ\times\times\times\circ\circ\ldots$, where the leftmost $\times$ occupies the zero position. 
The weight diagram of $0$ is 
$$\begin{array}{lll}
&\ \ \  \times^n\circ\circ\ldots  & \text{ for }  \osp(2n|2n),\fq_{2n}\\
&-\times^n\circ\circ\ldots   & \text{ for }  \osp(2n+1|2n)\\
&\ \ \ \overset{\times^n}{>}\circ\circ\ldots   & \text{ for }  
\osp(2n+2|2n),\fq_{2n+1}.
\end{array}$$

The diagram $+\circ\times\times$ 
corresponds to the $\osp(4|4)$-weight $\lambda=\lambda+\rho=(\vareps_2+\delta_2)+2(\vareps_1+\delta_1)$. The diagram $+\times^3$ 
corresponds to $\osp(7|6)$-weight
 $\lambda=\vareps_1$.

The empty diagram correspond to one of the algebras 
 $\fgl(0|0)=\osp(0|0)=\osp(1|0)=\fq_0=0$; the diagram
$>$  corresponds to the weight $0$  for $\osp(2|0)=\mathbb{C}$
or for $\fq_1$ (in both cases the corresponding simple 
highest weight module
is one-dimensional).

\subsubsection{Remark}\label{RemOSP}
By~\cite{ES1}, Proposition 4.11 for $\lambda\in\cB_0$
the simple $OSP(2n|2n)$-module is
either of the form $L(\lambda)$ if $\lambda_n=0$ or $L(\lambda)\oplus L(\lambda^{\sigma})$, where  $\lambda^{\sigma}$ is obtained from
$\lambda$ by changing the sign of $\xi$. Thus the simple $OSP(2n|2n)$-modules are
in one-to-one correspondence with the unsigned
$\osp(2n|2n)$-diagrams.

\subsection{Arch diagrams}
\label{arcs}
A {\em generalized  arch diagram} is the following data: 
\begin{itemize}
\item[$\bullet$]
a weight diagram $f$, where
the symbols $\times$ at the zero position are drawn vertically and
 $>$ (if it is present) is drawn in the bottom,
\item[$\bullet$]
a collection of non-intersecting arches, where each arch  is 
\begin{itemize}
\item
either $\arc(a;b)$  connecting
 the symbol $\times$  with $\circ$
at the position $b>a$;
\item
or $\arc(0;b,b')$  connecting
 the symbol $\times$ at the zero position with two  symbols $\circ$
at the positions $0<b<b'$;
\item for $\fq_{2n+1}$-case  $\arc(0;b)$ connecting $>$ (at the zero position) with $\circ$
at the positions $b>0$; this arch is called {\em wobbly}. {\footnote
{wobbly arches are important for the description of $\DS_x(L)$; we will not use them 
in our text.}}
\end{itemize}

\end{itemize}
An empty position  is called {\em free} 
 if this position is not an end of an arch; we say that
 $\arc(a;b)$ is  a {\em two-legged arch originated at} $a$ and 
$\arc(0;b,b')$ is a {\em three-legged arch originated at} $0$.
A generalized arch diagram is called {\em arch diagram} if  

\begin{itemize}
\item[$\bullet$]
each symbol $\times$ is the left end of exactly one arch;

\item[$\bullet$]
for $\fq_{2n+1}$-case the symbol $>$  is the left end of a wobbly arch;

\item[$\bullet$]
there are no free positions under the arches;

\item[$\bullet$]
for the $\fgl$-case all arches are two-legged;

\item[$\bullet$]
for the $\osp(2n|2n),\osp(2n+1|2k)$-cases the lowest $\times$ at the zero position
supports a two-legged arch and each other symbol $\times$ at the zero position supports a three-legged arch;

\item[$\bullet$]
for the $\fq_{2n+\ell},\osp(2n+2|2k)$-cases  each symbol $\times$  at the zero position supports a three-legged arch.
\end{itemize}

Each weight diagram $f$ admits a unique arch diagram
which we denote by $\Arc(f)$; this diagram can be constructed
in the following way:
we pass from right to left through the weight diagram and connect each symbol $\times$ with the next empty symbol(s) to the right by an arch.

\subsubsection{Partial order}
We consider a partial order on the set of arches by saying that one arch is smaller than
another one if the first one is "below" the second one:
$$\begin{array}{l}
\arc(a;b)>\arc(a';b')\ \ \Longleftrightarrow\ \ a<a'<b\\
\arc(0;b_1,b_2)>\arc(a';b')\ \ \Longleftrightarrow\ \ a'<b_2,\\
\arc(0;b_1,b_2)>\arc(0;b_1',b_2')\ \ \Longleftrightarrow\ \ b_2>b_2' 
\ \Longleftrightarrow\ \ b_1>b_1'.
\end{array}$$

\subsection{Map $\tau$}\label{tau}   
Following~\cite{GS}, we introduce a bijection
 $\tau$ between  the weight diagrams for $\osp(2n+2|2n)$ and $\osp(2n+1|2n)$:
for a $\osp(2n+2|2n)$-diagram $f$ we construct
 $\tau(f)$ by the
following procedure:
\begin{itemize}
\item[-]
we remove $>$ and then
shift all entires at the non-zero positions of $f$ by one position to the left; 
\item[-]
 we  add 
the sign $+$ if $f(1)=\times$ and
the sign
$-$ if $f(1)=\circ$ and $f(0)\not=>$.
\end{itemize}
For instance, $\tau(\overset{\times}{>})=-\times\ $, $\tau(>\times)=+\times\ $,  $\tau(\overset{\times}{>}\circ\times)=-\times\times\ $,
$\tau(>\circ \times)=\circ \times$.

One readily sees that $\tau$ is a one-to-one correspondence between
weight diagrams and that
there is a natural bijection between the arches in $\Arc(f)$ and $\Arc(\tau(f))$:
the image of $\arc(a;b)$ is $\arc(a-1;b-1)$, the image of $\arc(0;b_1,b_2)$
is $(0;b_1-1,b_2-1)$ if $b_1\not=0$ and $(0;b_2-1)$ if $b_1=0$;
 this bijection preserves  the partial order of the arches.

We will also denote by $\tau$ the corresponding bijection
between the weight (i.e., between the sets $\cB_0$ defined for $\osp(2n+2|2n)$
and $\osp(2n+1|2n)$).

\subsection{The algebras $\fg_{(s)}$}\label{gp}
For $\fg=\osp(2n+t|2n)$ we consider the chain
$$\osp(t|0)\subset\osp(2+t|2)\subset\osp(4+t|4)\subset\ldots\subset\osp(2n+t|2n)=\fg$$
where $\osp(2p+t|2p)$ corresponds to the last $2p+[\frac{t}{2}]$ roots in $\Sigma$; 
we denote the subalgebra $\osp(2s+t|2s)$ by $\fg_{(s)}$. Note that
$\fg_{(0)}=0$ for $t=0,1$ and
$\fg_{(0)}=\mathbb{C}$ to $t=2$.

Similarly, for $\fg=\fgl(n|n),\fq_{n+\ell}$ we consider the chains
$$\begin{array}{lc}
0=\fgl(0|0)\subset\fgl(1|1)\subset \ldots \subset \fgl(n|n) &\ \  \ \ \ 
\fq_{\ell}\subset\fq_{2+\ell}\subset\ldots \subset\fq_{2n+\ell}\end{array}$$
where for $i>0$ the subalgebras $\fgl(i|i)$
(resp., $\fq_{2i+\ell}$) corresponds to the middle 
$2i+\ell-1$ roots in $\Sigma$; we denote the subalgebra 
$\fgl(s|s)$ (resp., $\fq_{2s+\ell}$) by $\fg_{(s)}$. 
It is easy to see that for each $s$ there exist $z_s\in\ft$ such that
$\fg^{z_s}=\fg_{(s)}+\fh$.  

\subsubsection{}\label{tail}
We retain notation of~\ref{zii}.
For $\lambda\in \ft^*$ we denote by $\tail(\lambda)$ the maximal $i$ such that
$\lambda|_{\ft^{(i)}}=0$. 
 If $\rho=0$ (i.e., for $\fg=\osp(2n|2n),\osp(2n+2|2n)$ 
and $\fq_{2n+\ell}$) then $\tail(\lambda)$ is equal to 
  the number of $\times$ at the zero position of the weight diagram (and is equal to
the number of zeros among $\{\lambda_i\}_{i=1}^n$).  The map $\tau$ defined in~\ref{tau}
preserves the function $\tail$.

\section{Multiplicities $K^i(\lambda;\nu)$}\label{Section3}
We retain notation of Section~\ref{arches} and
set $\fp:=\fg_{(n-1)}+\fb$. The multiplicities 
$K^i(\lambda;\nu)$ were obtained in~\cite{PS1},\cite{PS2},\cite{MS} and~\cite{GS}. Below we will describe these multiplicites 
 in terms of arch diagrams.  
We introduce a Poincar\'e polynomial $K^{\lambda,\nu}(z)$ by
$$K^{\lambda,\nu}(z):=\sum_{i=0}^{\infty} K^i(\lambda;\nu) z^i=
\sum_{i=0}^{\infty}[\Gamma_i^{\fg,\fp}(L_{\fp}(\lambda)):
L_{\fg}(\nu)]z^i$$
(by~\cite{Penkov}, the sum is finite). One has
$K^{\lambda;\nu}(z)=0$ if $\lambda\in\cB$ and $\nu\not\in\cB$.
The polynomials $K^{\lambda,\nu}(z)$ for 
$\lambda,\nu\in\cB$ are given in Propositions~\ref{Kgl}, \ref{Kosp}, \ref{Kq}. \Prop{Kgl} ($\fgl$-case) is a simple reformulation of  Corollary 3.8 in~\cite{MS}. \Prop{Kosp} ($\osp$-case) is a reformulation of Proposition 7 in~\cite{GS}
(we translate the  formulae from~\cite{GS} to the language of arch diagrams). 
For the $\fq$-case the polynomials were described recursively by V.~Serganova 
and I.~Penkov in~\cite{PS1},  \cite{PS2}; in~\Prop{Kq} we present non-recursive formulae, which are deduced from the Penkov-Serganova recursive formulae. The rest of the section is occupied by  examples and the proof of~\Prop{Kq}.

\subsection{Notation}\label{fgreserved}
Let $g$ be a weight diagram.
We  denote by  $(g)_p^q$ the diagram
which obtained from $g$ by moving  $\times$ from the position $p$ to a free position $q>p$; such diagram is defined only if
$$ g(p)\in \{\times^i,\overset{\ \ \times^{i}}{>}\} \text{ for }i\geq 1\
\ \text{ and }\ \  
g(q)=\circ.$$ 
For instance, for $g=\times^2\circ\times$ one has $(g)_0^3=\times\circ\times\times$ and $(g)_0^2, (g)_1^5$ are not defined.
 If $g(0)=\times^i$
or $\overset{\ \ \times^{i}}{>}$ for $i>1$, we denote by $(g)_{0,0}^{p,q}$ 
the diagram
which obtained from $g$ by moving two symbols $\times$ from the zero position 
 to  free positions $p$ and $q$ with $p<q$; for example, $(\times^2\times)_{0,0}^{3,4}=\circ\times\circ\times\times$.

If $f(p)\not=\circ$, we denote by $\arc_f(p)$ the  positions 
``connected with $p$ in $\Arc(f)$''; for example, 
$\arc_{\times\circ \times}(2)=3$,
$\arc_{\times\circ \times}(0)=\{1,4\}$ for $\fq_{4}$ 
and  $\arc_{\times\circ \times}(0)=\{1\}$ for $\osp(4|4)$.
Notice that if $(f)_p^q$ is defined, then
$\arc_f(p)$ is defined.

We always assume that $\lambda,\nu\in\cB$ and denote by
 $g$ (resp., $f$) 
the weight diagram of $\lambda$ (resp., of $\nu$); we sometimes write
$K(\frac{g}{f})$ instead of $K^{\lambda,\nu}$.
As in~\ref{lambda+rho} let  $\lambda_1$ be the coordinate of the
rightmost symbol $\times$ in $g$.

\subsection{Proposition (see~\cite{MS}, Corollary 3.8)}\label{Kgl}
Take $\fg=\fgl(n|n)$. 
If $K^{\lambda,\nu}(z)\not=0$, then $g=(f)_a^{\lambda_1}$ and
$$K^{\lambda,\nu}(z)=\left\{\begin{array}{ll}
z^{b-\lambda_1} &\ \text{ if }\lambda_1\leq b\\
0  &\ \text{ if }b< \lambda_1\end{array}\right.$$
where $b:=\arc_f(a)$.

\subsection{Proposition (see~\cite{GS}, Proposition 7)}\label{Kosp}
Take $\fg=\osp(2n+t|2n)$ for $t=0,2$ and $\lambda\not=0$.

\begin{enumerate}
\item 
If $K^{\lambda,\nu}(z)\not=0$, then $g=(f)_a^{\lambda_1}$ or $g=(f)_{0,0}^{p,\lambda_1}$ and $f,g$ do not have different signs.
\item
Let $g=(f)_a^{\lambda_1}$ and $f,g$ do not have different signs. 

Set $b:=\max\arc_f(a)$  and
$b_-:=\min\arc_f(0)$ if $a=0$.

If  $a\not=0$ or $a=0$ and $t=2$, then
$\ \ K^{\lambda,\nu}(z)=\left\{\begin{array}{ll}
z^{b-\lambda_1} &\ \text{ if } \lambda_1\leq b\\
0  &\ \text{ if }b< \lambda_1.\end{array}\right.$

If $a=0$ and $t=0$, then
$$K^{\lambda,\nu}(z)=\left\{\begin{array}{ll}
z^{b_--\lambda_1}+z^{b-\lambda_1} &\ \text{ if }\lambda_1\leq b_-<b\\
z^{b-\lambda_1} &\ \text{ if }\lambda_1\leq b_-=b\\
z^{b-\lambda_1} &\ \text{ if }b_-< \lambda_1\leq b\\
0  &\ \text{ if }b< \lambda_1.\end{array}\right.$$

\item
Let $g=(f)_{0,0}^{p,\lambda_1}$. If $\Arc(f)$ contains $\arc(0;p,q)$ with $\lambda_1\leq q<\max\arc_f(0)$,
then  $K^{\lambda,\nu}(z)=z^{q-\lambda_1}$; otherwise
 $K^{\lambda,\nu}(z)=0$.
\item For $\lambda\not=0$ the polynomials $K^{\lambda,\nu}(z)$ for $\osp(2n+1|2n)$ can
by obtained from the polynomials for $\osp(2n+2|2n)$ by the formula
$K^{\tau(\lambda),\tau(\nu)}(z)=K^{\lambda,\nu}(z)$.
\end{enumerate}

\subsubsection{Examples} \begin{itemize}
\item[(1)]
For $\lambda=\vareps_1+\delta_1$ and $\nu=0$ one has $g=(f)_0^1$
with $b=2n$ for  $\osp(2n+2|2n)$
and $b=2n-1$, $b_-=1$ for $\osp(2n|2n)$.  
The polynomial $K^{\vareps_1+\delta_1,0}(z)$ equals to $1$ for $\osp(2|2)$,  to $1+z^{2n-2}$ for $\osp(2n|2n)$ with $n>1$ and  to $z^{2n-1}$ for $\osp(2n+2|2n)$.

\item[(2)]
Take $\fg=\osp(4|4)$ with $\nu=\vareps_1+\delta_1$. Then
$f=\times\times$ so
$$\Arc(f)=\{\arc(1;2),\arc(0;3)\},\ \ \ \arc_f(1)=\{2\},\ \ 
\arc_f(0)=\{3\}.$$
 The non-zero values of $K(\frac{g}{f})$ are given by the following table
$$\begin{array}{|l||l|l|l|l|}
\hline
g  &\times\circ\times=(\times\times)_1^2 & 
\circ\times\times=(\times\times)_0^2
&  \circ\times\circ\times=(\times\times)_0^3 \\
\hline
K(\frac{g}{f}) & 1 & z & 1 \\
\hline
\end{array}$$

\item[(3)]
Take $\fg=\osp(6|4)$ with $\nu=\vareps_1+\delta_1$.
Then $f=\overset{\times}{>}\times$ so
$$\Arc(f)=\{\arc(1;2),\arc(0;3,4)\},\ \ \ 
\arc_f(1)=\{2\},\ \ \arc_f(0)=\{3,4\}.$$  
The non-zero values of $K(\frac{g}{f})$ are given by the following table
$$\begin{array}{|l||l|l|l|l|l|}
\hline
g  &\overset{\times}{>}\circ\times=(f)_1^2 & 
>\times\times=(f)_0^2
&  >\times\circ\times & >\times\circ\circ\times=(f)_0^4 \\
\hline
K(\frac{g}{f}) & 1 & z^2 & z & 1\\
\hline
\end{array}$$

\item[(4)]
Take $\fg=\osp(6|6)$ with  $\nu=\vareps_1+\delta_1$. Then
$f=\times^2\times$ so
$$\Arc(f)=\{\arc(1;2),\arc(0;3),\arc(0;4,5)\},\ \  \arc_f(1)=\{2\},\  \arc_f(0)=\{3,4,5\}.$$ 
The non-zero values of $K(\frac{g}{f})$ are given by the following table
$$\begin{array}{|l||l|l|l|l|l|}
\hline
g  &\times^2\circ\times & \times\times\times
&  \times\times\circ\times &\times\times\circ\circ\times & \times\times\circ\circ\circ\times\\
\hline
K(\frac{g}{f}) & 1 & z+z^3 & 1+z^2 &  z & 1\\
\hline
\end{array}$$

\item[(5)]
Take $\fg=\osp(10|10)$ with  $f=\times^3\circ\circ\times\times$. Then
$$\Arc(f)=\{(\arc(4;5),\arc(3;6),\arc(0;1),\arc(0;2,7),\arc(0;9,10)\}.$$
For $g=  \times^3\circ\circ\times\circ\times,\ 
\times^3\circ\circ\circ\times\circ\times,\ $
one has $K(\frac{g}{f})=1$. In addition, 
$$\begin{array}{|l||l|l|l|}
\hline
g  
& \times\circ\times\times\times\times=(f)_{0,0}^{2,5}
& (f)_{0,0}^{2,6} &
\times\circ\times\times\times\circ\circ\times=(f)_{0,0}^{2,7}
 \\
\hline
K(\frac{g}{f})  & z^2 &z & 1\\
\hline
\end{array}$$
For $g=\times^3\circ\circ\circ\times\times$
one has $K(\frac{g}{f})=z$; for $g=(f)_0^i$ 
with $i=5,6,\ldots,10$  we have $K(\frac{g}{f})=z^{10-i}$.
Since $K(\frac{g}{f})\not=0$ implies $\lambda_1>\nu_1=4$
we get $K(\frac{g}{f})=0$ for other values of $g$.
\end{itemize}

\subsection{}\label{Kq}
\begin{prop}{propKa}
Take $\fg=\fq_{m}$ and $\lambda,\nu\in\cB_0$ or $\lambda,\nu\in\cB_{1/2}$.

\begin{enumerate}
\item One has 
$K^{0,0}(z)=z+z^2+\ldots+z^{m-1}$ and  $K^{\lambda,\nu}(z)=0$
for $\nu\not=0$.

\item
If $\lambda\not=0$ and $K^{\lambda,\nu}(z)\not=0$, 
 then
$g=(f)_a^{\lambda_1}$ for $a<\lambda_1$.

\item Let $g=(f)_a^{\lambda_1}$ for $a<\lambda_1$.
Set $b:=\max\arc_f(a)$.

If $a\not=0$, then 
 $K^{\lambda,\nu}(z)=\left\{\begin{array}{ll}
z^{b-\lambda_1} &\ \text{ if }\lambda_1\leq b\\
0  &\ \text{ if }b< \lambda_1.\end{array}\right.$

If $a=0$, set $A_{f;\lambda_1}:=\{i\in \arc_f(0)|\ \lambda_1\leq i<b\}$. Then
$$K^{\lambda,\nu}(z)=\left\{\begin{array}{ll}
0  &\ \text{ if }\ \ A_{f;\lambda_1}=\emptyset\\
z^{i_--\lambda_1}+z^{i_+-\lambda_1}\  &\ \text{ otherwise}\end{array}\right.$$
 where 
$i_-:=\min A_{f;\lambda_1}$, $i_+:=\max A_{f;\lambda_1}$.
\end{enumerate}
\end{prop}

\subsubsection{Examples}
In the examples below we compute  $K^{\lambda,\nu}(z)$
using~\Prop{propKa}.

\begin{itemize}
\item[(1)] For $\lambda=\vareps_1-\vareps_m$ and $\nu=0$ one has
$g=(f)^1_0$ with $\arc_f(0)=\{1,\ldots,m\}$ and thus $A_{f;1}=\{1,\ldots,m-1\}$.
This gives $K^{\vareps_1-\vareps_m,0}=1+z^{m-2}$ as in~\cite{PS1}, Theorem 4.

\item[(2)]
Take $\fg=\fq_4$ and $f=\times\times$. Then 
$$\Arc(f)=\{\arc(1;2);\ \arc(0; 3,4)\}$$
and
$\arc_f(1)=\{2\},\ \arc_f(0)=\{3,4\}$. This gives

$$\begin{array}{|l||l|l|l|}
\hline
g  &\times\circ\times=(\times\times)^2_1 & \circ\times\circ\times=(\times\times)^3_0 &
\circ\times\times=(\times\times)^2_0 \\
\hline
K(\frac{g}{f}) & 1 & 2& 2z  \\
\hline
\end{array}$$
and $K(\frac{g}{f})=0$ for other values of $g$.

\item[(3)]
Take $f=\times^2\times\circ\circ \times$. One has 
$\arc_f(0)=\{3,6,7,8\}$ and
$$\begin{array}{|l||l|l|l|l|}
\hline
g  &\times\times\circ\circ\times\times=(f)_0^5 &
(f)_0^6 &
\times\times\circ\circ\times\circ\circ\times=(f)_0^7  & (f)^5_4
 \\
\hline
K(\frac{g}{f}) & z+z^2 & 1+z & 2 & 1\\
\hline
\end{array}$$
 Since $K(\frac{g}{f})=0$ implies $\lambda_1>\nu_1=4$ we get $K(\frac{g}{f})=0$ for other values of $g$.

\item[(4)]
Take $f=\times^2\times\circ\circ\circ \times$. One has $\arc_f(0)=\{3,4,7,8\}$
which gives
$$K\bigl(\frac{(f)_0^6}{f}\bigr)=2z,\ \ K\bigl(\frac{(f)_0^7}{f}\bigr)=2,\ \ K\bigl(\frac{(f)_5^6}{f}\bigr)=1$$
Since $K(\frac{g}{f})=0$ implies $\lambda_1>\nu_1=5$ we get $K(\frac{g}{f})=0$ for other values of $g$.

\end{itemize}

\subsection{Proof of~\Prop{propKa}}\label{proofpropka}
Theorem 4 in~\cite{PS1} gives  (i) and
establishes (ii), (iii) for $m=1$ (in this case  $\cB=\{0\}$).
From now on we assume that $m\geq 2$ and $\lambda\not=0$. We set
$$\theta:=\vareps_1-\vareps_m.$$

\subsubsection{Notation}
Recall that $m=2n+\ell$ and $n>0$. For $\mu\in\cB$ we 
write $\mu=(\mu_1,\ldots,\mu_n)$ and
set $\mu':=\mu|_{\ft_{(n-1)}}$, i.e., $\mu'=(\mu_2,\ldots,\mu_n)$.
We will denote the weight diagram of $\mu$ by $\diag(\mu)$.
For a polynomial $P\in\mathbb{Z}[z]$ we introduce $\ol{P}\in\{0,1\}$
by $\ol{P}:=P(0)$
 modulo $2$; we will also use the following notation:
$\bigl(\sum_{i=-\infty}^{\infty} d_iz^i\bigr)_+:=\sum_{i=0}^{\infty} d_iz^i$.

\subsubsection{Formulae from~\cite{PS1}}
Theorem 4 in~\cite{PS1} can be written in the following form
\begin{equation}\label{eqtheta}
\begin{array}{lcl}
K^{\frac{\theta}{2},\nu}=0\ \text{ for }m=2 & & K^{\theta,\mu}=\delta_{0,\mu}(1+z^{m-2}). \end{array}
\end{equation}

Theorem 3 in~\cite{PS1} gives for $m\geq 2$, $\lambda_1>1$ and
$\nu\not=\lambda-\theta$ 
\begin{equation}\label{recursive}
\begin{array}{lll}
K^{\lambda,\lambda-\theta}=1,\\
K^{\lambda,\nu}=(z^{-1}K^{\lambda-\theta,\nu})_+\  & \text{ for }&\lambda_1>\lambda_2+1,\ 
 \  \tail(\nu)\leq \tail(\lambda)\\
K^{\lambda,\nu}=(z^{-1}K^{\lambda-\theta,\nu})_++\ol{K^{\lambda-\theta,\nu}} \ & \text{ for }& \lambda_1>\lambda_2+1,\ 
 \  \tail(\nu)>\tail(\lambda)\\
K^{\lambda,\nu}=0& \text{ for }&\lambda_1=\lambda_2+1, \ \nu_1\not=\lambda_2\\
K^{\lambda,\nu}=zK^{\lambda',\nu'}& \text{ for }&\lambda_1=\lambda_2+1, \ \nu_1=\lambda_2.
\end{array}\end{equation}

\subsubsection{Case $\lambda_1\leq 1$}\label{theta}
In this case
$\lambda=0,\theta$ or $\lambda=\frac{\theta}{2}$ for $m=2$
(note that $\frac{\theta}{2}\not\in P^+(\fq_m)$ for $m>2$).
For  $m=2$  there is 
no diagram $f$ satisfying $(f)_a^b=\diag(\frac{\theta}{2})$.
If $(\diag(\nu))_{a}^{b}=\diag(\theta)$, then $a=0$ and $\nu=0$,
so $\arc_{\diag(\nu)}(0)=\{1,2,\ldots,m\}$. Comparing this with~(\ref{eqtheta}) 
we obtain (ii), (iii) for the case $\lambda_1\leq 1$.

\subsubsection{Case $n=1$}
In this case $\cB_0=\mathbb{N}\theta$ and $\cB_{1/2}=(\mathbb{N}+\frac{1}{2})\theta$ for $\ell=0$.
The induction on $\lambda_1$  gives
$K^{\lambda,\nu}(z)=\delta_{\lambda-\theta,\nu}$ (for $\lambda\not=0$);
this gives (ii), (iii) for the case $n=1$.

\subsubsection{}
If $\nu=\lambda-\theta$, then 
$\diag(\lambda)=(\diag(\nu))_{\lambda_1-1}^{\lambda_1}$ and $K^{\lambda,\nu}(z)=1$
by~(\ref{recursive}); thus (iii) holds for this case.

\subsubsection{}
Assume that $\nu\not=\lambda-\theta$.
Set  $j:=\lambda_1-\lambda_2-1$ and take
$\mu:=\lambda-j\theta$
(i.e., $\diag(\mu)$ is obtained from $\diag(\lambda)$ by moving the 
rightmost $\times$ to the left ``as much as possible'':
for instance, if $\diag(\lambda)=\times\circ\times\circ\circ\times$,
then  $\diag(\mu)=\times\circ\times\times$). 
By~(\ref{recursive}), $K^{\lambda,\nu}(z)\not=0$
 implies $K^{\mu,\nu}(z)\not=0$ which forces 
$\nu_1=\mu_2$ (since $\mu_1-1=\mu_2$). Hence $\nu_1=\lambda_2$.
We obtain
$$K^{\lambda,\nu}(z)\not=0,\ \ \nu\not=\lambda-\theta \ \ \Longrightarrow\ \ \ 
\nu_1=\lambda_2.$$

\subsubsection{}
We will prove (ii), (iii) by the induction on $\lambda_1$ (note that $\lambda_1\geq \frac{1}{2}$
since  $\lambda\not=0$). The cases $\lambda_1\leq 1$ and $n=1$ are 
established above. 
From now till the end of the proof   we assume 
$$n\geq 2,\ \ \ \ \lambda_1>1,\ \ \ \nu_1=\lambda_2,\ \ \ \nu\not=\lambda-\theta.$$

Using $\nu_1=\lambda_2$ we write  $\diag(\lambda),\diag(\nu)$ in the form
\begin{equation}\label{ftimesg}
\diag(\lambda)=g*\times,\ \ \ \diag(\nu)=f *\circ
\end{equation}
 where the symbols $*\in\{\circ,\times\}$ occupies the
position $\lambda_1-1$ in both diagrams (note that $f,g$ do not have the same meaning as in~\ref{fgreserved}).
For example, 
$$\begin{array}{rlccl}
\diag(\lambda)=&\circ\times\times\times\circ\times\times  &\ \  & \diag(\nu)=& \times\times\times\times\circ\times\circ \\
\ *=\times,\ \ \ \ g=& \circ\times\times\times\circ\ \ \ \  & & 
f=& \times\times\times\times\circ
\end{array}$$
The formulae~(\ref{recursive}) give
\begin{equation}\label{Kfg}
\begin{array}{ll}
K(\frac{f\circ\times}{f\times\circ})=1\\
K(\frac{g\circ\times}{f\circ\circ})=(z^{-1}K(\frac{g\times}{f\circ}))_+ & \text{ if }\ \  \tail(\nu)\leq \tail(\lambda),\\ K(\frac{g\circ\times}{f\circ\circ})=(z^{-1}K(\frac{g\times}{f\circ}))_+ +\ol{K(\frac{g\times}{f\circ})}& \text{ if }\ \  \tail(\nu)>\tail(\lambda),\\
 K(\frac{g\times\times}{f\times\circ})=zK(\frac{g\times}{f\circ}),\ \ \ \  \ \ \ K(\frac{g\times\times}{f\circ\circ})=0.
\end{array}\end{equation}

\subsubsection{Proof of  (ii)}
Assume that $K^{\lambda,\nu}(z)\not=0$.

If $*=\times$, then 
$K^{\lambda,\nu}=K(\frac{g\times\times}{f\times\circ})=zK(\frac{g\times}{f\circ})$. By induction,
$K(\frac{g\times}{f\circ})\not=0$ implies that $g\times=(f\circ)_a^{\lambda_1-1}$
for some $a$, which gives
$g\times\times=(f\times\circ)_a^{\lambda_1}$.

If  $*=\circ$, then $K^{\lambda,\nu}=K(\frac{g\circ\times}{f\circ\circ})\not=0$.  By~(\ref{Kfg}), this gives $K(\frac{g\times}{f\circ})\not=0$. By induction
this implies $g\times=(f\circ)_a^{\lambda_1-1}$ for some $a$,
which gives 
$g\circ\times=(f\circ\circ)_a^{\lambda_1}$. 

This establishes (ii).

\subsubsection{}\label{iiipf}
The proof of (iii) occupies~\ref{iiipf}---\ref{a0}. 
We assume that 
$\diag(\lambda)=(\diag(\nu))^{\lambda_1}_a$ and $\nu\not=\lambda-\theta$. Then
\begin{equation}\label{starno}\begin{array}{ccccc}
\diag(\lambda)=g * \times & & \diag(\nu)=f * \circ & & g\times=(f\circ)_a^{\lambda_1-1}.\end{array}\end{equation}

\subsubsection{Case $a\not=0$}
In this case $\tail(\nu)=\tail(\lambda)$. Take
 $b':=\arc_f(a)$  and $b:=\arc_{f*}(a)$.

If $*=\circ$, then $f=f*$ and $b=b'$. By induction  we  get
$$K(\frac{g\circ\times}{f\circ\circ})=\bigl(z^{-1}K(\frac{g\times}{f\circ})\bigr)_+=\bigl(z^{-1}(z^{b-(\lambda_1-1)})_+\bigr)_+=(z^{b-\lambda_1})_+
$$
as required. For $*=\times$ one has
 $\arc(\lambda_1-1;\lambda_1)\in\Arc(f\times\circ)$, so $b=b'$
if $b'<\lambda_1-1$ and $b=b'+2$ otherwise.
By induction we get
$$K(\frac{g\times\times}{f\times\circ})=zK(\frac{g\times}{f\circ})=
z (z^{b'-(\lambda_1-1)})_+=(z^{b-\lambda_1})_+.$$
This establishes the required formula for $a\not=0$.

\subsubsection{Case $a=0$}\label{a0}
In this case $\tail(\nu)=\tail(\lambda)+1$. 
Set 
$$i_-:=\min A_{f*;\lambda_1},\ \ i_+:=\min A_{f*;\lambda_1},\ \ 
i'_-:=\min A_{f;\lambda_1-1},\ \ i'_+:=\min A_{f;\lambda_1-1}$$
taking $i_{\pm}=-\infty$ (resp., $i'_{\pm}=-\infty$)
if $A_{f*;\lambda_1}=\emptyset$ (resp., $A_{f;\lambda_1-1}=\emptyset$).
By induction
$$K(\frac{g\times}{f\circ})=z^{i'_--(\lambda_1-1)}+z^{i'_+-(\lambda_1-1)}.$$

If $*=\times$, then $i_{\pm}=i'_{\pm}+2$ and~(\ref{Kfg}) gives
$$K(\frac{g\times\times}{f\times\circ})=zK(\frac{g\times}{f\circ})=
z^{i_--\lambda_1}+z^{i_+-\lambda_1}.$$

Consider the remaining case $*=\circ$. 
By~(\ref{Kfg}) we have
$$K(\frac{g\circ\times}{f\circ\circ})=\bigl(z^{-1}K(\frac{g\times}{f\circ})\bigr)_+
+\ol{K(\frac{g\times}{f\circ})}.$$

Since
the  coordinates of $\times$ in $f$ are smaller than $\lambda_1-1$,
$\arc_f(0)$ contains all integers between $i'_-$ and $i'_+$.
Thus $\lambda_1-1\leq i'_-\leq i'_+$ and
$$A_{f*,\lambda_1}=A_{f,\lambda_1}=\{i|\ i\not= \lambda_1-1,\ i'_-\leq i\leq i'_+ 
\}.
$$

If $i'_-\not=\lambda_1-1$, this gives $i_-=i'_-$ and $i_+=i'_+$ which imply
$$K(\frac{g* \times}{f *\circ})=z^{-1} 
K(\frac{g\times}{f\circ})=z^{i_--\lambda_1}+z^{i_+-\lambda_1}.$$

If $i'_-=i'_+=\lambda_1-1$, then $A_{f*,\lambda_1}=\emptyset$
and thus $i_{\pm}=-\infty$. One has
$K(\frac{g\times}{f\circ})=2$, so 
$K(\frac{g* \times}{f *\circ})=0=z^{i_--\lambda_1}+z^{i_+-\lambda_1}$.

If $i'_-=\lambda_1-1<i'_+$, then $i_-=\lambda_1$, $i_+=i'_+$. In this case
$K(\frac{g\times}{f\circ})=1+z^{i'_+-(\lambda_1-1)}$ and
$K(\frac{g* \times}{f *\circ})=1+z^{i_+-\lambda_1}$.

We see that in all cases $K(\frac{g*\times}{f*\circ})=
z^{i_--\lambda_1}+z^{i_+-\lambda_1}$. 
This completes the proof of (iii).
\qed

\section{The grading $\pari$ and the computation of $\ext(\lambda;\nu)$}\label{sectiondex}
In this section we introduce the $\mathbb{Z}_2$-grading $\pari$
and  describe the graphs $G(\cB;K^0)$. Then we describe the graphs
$(\CC;\ext)$ which were defined in  Introduction.

\subsection{The grading $\pari$ }
Recall that $\ell=1$ for $\osp(2n+2|2n)$, $\fq_{2n+1}$ and $\ell=0$ in other cases. 
For $\lambda,\nu\in\cB$ we take $\lambda_1,\ldots,\lambda_n$ as in~\ref{lambda+rho}
and introduce 
\begin{equation}\label{paridef}\begin{array}{l}
||\lambda||:=\left\{\begin{array}{ll}
\displaystyle\sum_{i=1}^n \lambda_i & \text{ if } \fg\not=\osp(2n+2|2n)\\
\displaystyle\sum_{i=1}^n \lambda_i-\ell(n-\tail\lambda) & \text{ if }\fg=\osp(2n+2|2n)  \end{array}\right.\\
\ \ \pari(\lambda):=||\lambda|| \mod 2\ \ \ \ \ \ 
 \pari(\lambda;\nu):=
\left\{\begin{array}{ll}
0 & \text{ if } \pari(\lambda)=\pari(\nu)\\
1 & \text{ if } \pari(\lambda)\not=\pari(\nu) \end{array}\right.\\
 \tail(\nu;\lambda):=\tail\nu-\tail\lambda.
\end{array}
\end{equation}

Observe that $||\tau(\lambda)||=||\lambda||$
for $\fg=\osp(2n+2|2n)$. 

Recall that we use $\equiv$ for 
the equivalence modulo $2$.

\subsubsection{}
\begin{cor}{cordex}
Let $\lambda,\nu\in \cB$ be such that  $K^{\lambda,\nu}(z)\not=0$.
\begin{enumerate}
\item
For $\fg=\fgl(n|n)$ or $\fg=\fq_{2n}$ with $\cB;=\cB_{1/2}$ one has
$\lambda>\nu$ and
$K^{\lambda,\nu}(z)=z^i$, where $i\equiv \pari(\lambda;\nu)+1$ modulo $2$.

\item
For $\fg=\osp(2n+2|2n),\osp(2n+1|2n)$ with $\lambda\not=0$, one has
   $\lambda>\nu$, $\tail(\nu;\lambda)$ in $\{0,1,2\}$ and
$K^{\lambda,\nu}(z)=z^i$ where $i\equiv \pari(\lambda;\nu)+1$.

\item
Take $\fg=\osp(2n|2n)$ with $\lambda\not=0$.  Then  $\lambda>\nu$ and $\tail(\nu;\lambda)$ in $\{0,1,2\}$.
If $\tail(\nu;\lambda)\not=1$, then $K^{\lambda,\nu}(z)=z^i$; if $\tail(\nu;\lambda)=1$, then
$K^{\lambda,\nu}(z)$ equals to $z^i$ or to $z^i+z^{j}$
with $j<i$  and $j\equiv i$ modulo $2$. In both cases
 $i\equiv \pari(\lambda;\nu)+1$.

\item
Take $\fg=\fq_{2n+\ell}$  with $\lambda\not=0$.
Then  $\lambda>\nu$ and $\tail(\nu;\lambda)\in\{0,1\}$.

If $\tail(\nu;\lambda)=0$, then $K^{\lambda,\nu}(z)=z^i$
for $i\equiv \pari(\lambda;\nu)+1$.

If $\tail(\nu;\lambda)=1$, then
$K^{\lambda,\nu}(z)=z^i+z^{j}$
with $j\leq i$ and $i\equiv \pari(\lambda;\nu)+1+\ell$.
\end{enumerate}
\end{cor}
\begin{proof}
By~\Prop{Kosp} (iv) we can assume $\fg\not=\osp(2n+1|2n)$.
Theorems~\ref{Kgl}--\ref{Kq} immediately imply all 
assertions except $i\equiv \pari(\lambda;\nu)+1$ modulo $2$
and $j\equiv i$ modulo $2$ for $\osp(2n|2n)$. 
We retain notation of~\ref{Kgl}--\ref{Kq}. Recall that  $K^{\lambda,\nu}(z)\not=0$
implies
$g=(f)_a^{\lambda_1}$ or $\fg=\osp(2n+t|2n)$ and $g=(f)_{0,0}^{p,\lambda_1}$.

Consider the case $g=(f)_a^{\lambda_1}$.
 In this case 
$$\pari(\lambda;\nu)\equiv \left\{\begin{array}{ll}
\lambda_1-a & \text{ if } a\not=0\ \text{ or } \fg\not=\osp(2n+2|2n),\fq_{2n+1}\\
\lambda_1-a+1 & \text{ if }  a=0\ \text{ and } \fg=\osp(2n+2|2n),\fq_{2n+1}.
\end{array}
\right.$$

 Consider the case when $\fg\not=\fq_{2n+\ell}$ or $a\not=0$. In this case $i=b-\lambda_1$, where $b=\max \arc_f(a)$.
Observe that $b-a$ is odd except for the case when
 $\fg=\osp(2n+2|2n)$ and $a=0$; in the latter case
$b-a$ is even. Hence  $i\equiv \pari(\lambda;\nu)+1$ if $\fg=\fq_{2n+\ell}$ or $a\not=0$. For $\osp(2n|2n)$ with $a=0$ one has $j=b_--\lambda_1$, where
$b_-=\min \arc_f(0)$ is odd; this gives $j\equiv \pari(\lambda;\nu)+1$.

Consider the  case 
$\fg=\fq_{2n+\ell}$ with $a=0$. One has
$K^{\lambda,\nu}(z)=z^{i_--\lambda_1}+z^{i_+-\lambda_1}$,
where $i_-\leq i_+=\max \{s\in\arc_f(0)|\ \lambda_1\leq s<\max \arc_f(0)\}$. Observe that
 $i_+\equiv \ell+1$, so $i_+-\lambda_1\equiv\pari(\lambda;\nu)+1+\ell$
as required.

For the remaining case 
$\fg=\osp(2n+t|2n)$ and $g=(f)_{0,0}^{p,\lambda_1}$ one has
$\pari(\lambda;\nu)\equiv p+\lambda_1$ modulo $2$.
In this case $i=q-\lambda_1$, where
 $\arc(0;p,q)$ is a three-legged arch in $\Arc(f)$. 
Since  $q-p$ is odd, this implies $i\equiv \pari(\lambda;\nu)+1$. This completes the proof.
\end{proof}

\subsubsection{Remark}
The coefficients of the character formulae obtained in~\cite{GS},\cite{SZq},\cite{GH2} 
can be expressed in terms of the values
${K}^{\lambda,\nu}(-1)$.
By above, if ${K}^{\lambda,\nu}(-1)\not=0$, then 
$$(-1)^{\pari(\lambda;\nu)+1}{K}^{\lambda,\nu}(-1)=
\left\{\begin{array}{ll}
1 & \text{ for }  \fgl(n|n), \osp(2n+1|2n),\osp(2n+2|2n)\\
1 & \text{ for } \osp(2n|2n), \fq_{2n+\ell} \ \ \text{ if } \tail(\nu;\lambda)\not=1\\
1 \text{ or } 2 & \text{ for } 
\osp(2n|2n) \ \ \text{ if } \tail(\nu;\lambda)=1\\
(-2)^{\ell}& \text{ for } 
\fq_{2n+\ell} \ \  \text{ if } \tail(\nu;\lambda)=1.\\
\end{array}
\right.$$

 \subsection{Example}\label{expn=1}
For $n=1$ the polynomials $K^{\lambda,\nu}$ can be presented by the following graphs
where the arrows stands for $K^{\lambda,\nu}\not=0$ and the solid arrows
for $K^{\lambda,\nu}(0)$, so the solid arrows constitute the graph $G(\cB;K^0)$.
If $K^{\lambda;\nu}(z)$ is not a constant polynomial, we write 
 $K^{\lambda;\nu}(z)$ near the corresponding arrow. Using Remark~\ref{RemOSP} we obtain

$$\xymatrix{& \fgl(1|1): &\ldots\ar[r] &-\beta\ar[r]&0 \ar[r] & \beta\ar[r] & 2\beta\ar[r] 
& \ldots & \\
& \osp(2|2): &\ldots  &-\beta'\ar[l]&0 \ar[r]\ar[l] & \beta\ar[r] &2\beta\ar[r] 
& \ldots &\\
& OSP(2|2): & &  0 \ar[r] & \beta\ar[r] &2\beta\ar[r] 
& \ldots &\\}$$
$$\xymatrix{&\osp(4|2): &   &0\ar[r]\ar@{.>}[rd]_z
&2\beta\ar[r] &3\beta
\ar[r]&\ldots& \\
& & & &   \beta\ar[u]\ar@{.>}[lu]&&&
\\
 &\fq_2, \cB_{1/2}: &  &  0 \ar[r] & \frac{\theta}{2}\ar[r] &\frac{3\theta}{2}\ar[r] &\frac{5\theta}{2}\ar[r] 
& \ldots & \\ 
&\fq_2, \cB_0: &  &  0\ar@{=>}[r]& {\theta}\ar[r] &2\theta\ar[r] & 3\theta\ar[r] 
& \ldots & \\
& \fq_3, \cB_0: & & 0\ar@{=>}[r]\ar[rd]_{1+z} &2\theta\ar[r] &3\theta
\ar[r]&\ldots& \\
& & &  & \theta\ar[u]&&&
\\
}
$$ 

For $\fg\not=\osp(4|2)$ the grading $\pari$ is given by
$\pari(i\beta)\equiv i$, $\pari(i\beta')\equiv i$, $\pari(i\theta)\equiv i$;
for $\osp(4|2)$ one has  $\pari(i\beta)\equiv i-1+\delta_{i0}$.

\subsection{Polynomials $\hat{K}^{\lambda,\nu}(z;w)$}\label{Gammagsps}
Retain notation of~\ref{zi}, \ref{zii} and \ref{gp}. 
Substituting $\fg$ by $\fg_{(s)}$ 
we obtain the functors $\Gamma_{\bullet}^{\fg_{(s)},\fp_{(s)}}$
which satisfy the assumptions (A), (B) of~\ref{AiB}. 
The formulae for $K^i_{(s)}(\lambda;\nu)$ can be obtained
from the formulae for $K^i(\lambda;\nu)$ 
by changing $\lambda_1$ to
$\lambda_s$ and $m$ to $s$ in $\fq_m$-case.

\begin{equation}\label{hatK}
\hat{K}^{\lambda,\nu}(z;w):=\sum_{w=1}^k \sum_{i=0}^{\infty} 
K^j_{(i)}(\lambda;\nu)z^iw^j.\end{equation}
Using~\ref{Kgl}--\ref{Kq} we obtain 
 $\hat{K}^{\lambda,\lambda}(z,w)=0$ for any $\lambda\in\cB$ with $\tail\lambda=0$ (for $\fgl(n|n)$ this holds  for any $\lambda\in\cB$).

\subsubsection{}
\begin{cor}{cora1}
Take $\lambda\in \cB$.
\begin{enumerate}
\item 
For $\fg=\fgl(n|n)$ and $\fq_{2n+\ell}$ one has 
$$\hat{K}^{\lambda,\nu}(z,w)\not=0\ \ \Longrightarrow\ \ 
\lambda\geq \nu\ \ \& \ \ \nu\in\cB.$$

\item For $\osp(2n+t|2n)$ one has 
$$\hat{K}^{\lambda,\nu}(z,w)\not=0\ \ \&\ \ \lambda\geq\nu\ \ \Longrightarrow\ \ 
 \nu\in\cB.$$

\end{enumerate}
\end{cor}
\begin{proof}
Assume that $\hat{K}^{\lambda,\nu}(z;w)\not=0$ for some $\nu\not=\lambda$;
for $\osp(2n+t|2n)$ we assume, in addition, $\lambda>\nu$.

Since $\hat{K}^{\lambda,\nu}(z;w)\not=0$
one has ${K}^i_{(s)}(\lambda;\nu)\not=0$ for some $i,s$.
Set $\lambda':=\lambda|_{\ft_{(s)}}$, $\nu':=\nu|_{\ft_{(s)}}$ and let
$\cB'\subset P^+(\fg_{(s)})$ be the analogue of the set $\cB$ for
$\fg_{(s)}$. Note that $\lambda'\in  \cB'$. By above,

\begin{enumerate}
\item[(a)]
$K^i(\lambda';\nu')={K}^i_{(s)}(\lambda;\nu)\not=0$
\end{enumerate}
which implies 
\begin{enumerate}
\item[(b)] $\nu'\in P^+(\fg_{(s)})$ and
 $\nu|_{\ft_{(s)}^{\perp}}=\lambda|_{\ft_{(s)}^{\perp}}$.
\end{enumerate}
In particular, $\nu'\not=\lambda'$ (since $\nu\not=\lambda$ and 
$\nu|_{\ft_{(s)}^{\perp}}=\lambda|_{\ft_{(s)}^{\perp}}$).
In the $\osp$-case combining (b) and  $\nu<\lambda$ we obtain
$\nu'<\lambda'$;
since $0$ is the minimal element in $P^+(\osp(2s+t|2s))$ this implies $\lambda'\not=0$.
We conclude that $K^i(\lambda';\nu')$ is given by~\ref{Kgl}-- \ref{Kq} 
(since $\lambda'\not=0$ for the $\osp$-case). Using~\ref{Kgl}-- \ref{Kq} we deduce from (a)
\begin{enumerate}
\item[(c)]
$\nu'\in \cB'\ $ and $\ \nu'<\lambda'$
\end{enumerate}
for all cases. Combining $\nu'<\lambda'$ with (b) we obtain $\lambda>\nu$
for $\fgl(n|n)$ and $\fq_{2n+\ell}$.

Let us show that $\nu\in\cB$.
Combining (b) and (c) we conclude that
$\nu+\rho$ can be written in the form appeared in~\ref{lambda+rho}.
 Moreover~\ref{Kgl}-- \ref{Kq} 
give
\begin{enumerate}
\item[(d)]  $\nu_{n+1-s}<\lambda_{n+1-s}$
\end{enumerate}
Combining (b) and (c)
we conclude that $\nu_i$s are integral (resp., non-negative integral,
in $\mathbb{N}+1/2$) for $\fg=\fgl(n|n)$ (resp., for $\cB_0$ with
$\fg\not=\fgl(n|n)$, for $\cB_{1/2}$). By (b)
\begin{equation}\label{nuii}
\nu_i=\lambda_i\ \ \text{ for }1\leq i\leq n-s.\end{equation}
Since $\lambda\in\cB$ one has
$\lambda_{n+1-s}\leq \lambda_{n-s}=\nu_{n-s}$; using (d) we get
\begin{equation}\label{nus}
\nu_{n+1-s}<\nu_{n-s}.\end{equation}
For $\fgl(n|n)$-case and for $\fq_{2n}$ with $\cB_{1/2}$  combining (c),
 (\ref{nuii}), (\ref{nus}) and the condition $\lambda\in\cB$
 we get
$\nu_i<\nu_{i+1}$ for each $i$.
For other cases we get either $\nu_i\leq \nu_{i+1}$  or
$\nu_i=\nu_{i+1}=0$ for each $i$. This implies $\nu\in\cB$.
\end{proof}

\subsubsection{Example}
The following example shows that $\hat{K}^{\lambda,\nu}(z,w)\not=0$
does not imply $\lambda\geq \nu$ or $\nu\in\cB$  in $\fosp$-case.
By~\cite{Germoni2}, $K^{0,\vareps_1}(z)=z$
for $\osp(3|2)$; this implies
$\hat{K}^{0,\vareps_2}=zw$ for  $\osp(5|4)$ whereas $0<\vareps_2$ and
$\vareps_2\not\in P^+(\osp(5|4))$ (and  so $\vareps_2\not\in\cB$).

\subsubsection{}
\begin{cor}{cora2}
Take $\lambda\not=\nu\in \cB$ and set $s:=n+1-\max \{i|\ \lambda_i\not=\nu_i\}$.
\begin{enumerate}
\item Take $\fg=\fgl(n|n), \fq_{2n+\ell}$. If $\hat{K}^{\lambda,\nu}(z;w)\not=0$, then
$$\hat{K}^{\lambda,\nu}(z;w)=\left\{\begin{array}{lll}
z^iw^s &  \text{ for  $\fgl(n|n)$}\\
z^iw^s\ &  \text{ for  $\fq_{2n+\ell}$} & \text{ if }\ \ \tail\lambda=\tail\nu\\
(z^i+z^{j})w^s &  \text{ for  
$\fq_{2n+\ell}$} & \text{ if }\ \ \tail\lambda\not=\tail\nu\\
\end{array}
\right.$$
with $0\leq j\leq i$ in the last case.

\item Take $\fg=\osp(2n+t|2n)$. If $\hat{K}^{\lambda,\nu}(z;w)\not=0$
and $\tail(\lambda)\leq\tail(\nu)$, then
$$\hat{K}^{\lambda,\nu}(z;w)=\left\{\begin{array}{ll}
z^iw^s &  \text{ for  } \osp(2n+1|2n),\osp(2n+2|2n)\\
z^iw^s\ & \text{ for }\osp(2n|2n) \  \text{ if }\ \tail\lambda=\tail\nu\\
z^iw^s\ \text{ or }(z^i+z^{i-2i'})w^s\ &\text{ for  } \osp(2n|2n) \  
\text{ if }\ \tail\lambda\not=\tail\nu\\
\end{array}
\right.$$
with $0\leq i-2i'<i$ in the last case.
\end{enumerate}
In all cases 
$i\equiv\pari(\lambda)-\pari(\nu)+1$ modulo $2$.
\end{cor}
\begin{proof}
The formulae in~\ref{Kgl}, \ref{Kq} give (i). For (ii) take
 $\lambda',\nu'$ as in the proof of~\Cor{cora1}.
The conditions $\lambda\not=\nu$ and $\tail (\lambda)\leq\tail(\nu)$
imply $\lambda'\not=\nu'$ and $\tail (\lambda')\leq\tail(\nu')$ 
which force $\lambda'\not=0$. Therefore $K^i_{(s)}(\lambda;\nu)=K^i(\lambda';\nu')$ is given by~\ref{Kosp}; this gives  (ii).
\end{proof}

\subsection{Graph $G(\cB;K^0)$}
Retain notation of~\ref{graphs12}. 
By~\ref{cora1}, if $\nu\to\lambda$ is an edge in  $G(\ft^*,K^0)$ with $\lambda\in\cB$, then $\nu\in\cB$. In other words,
$B(\lambda)\subset\cB$ for each $\lambda\in\cB$. 
Using~\ref{Kgl}-- \ref{Kq} we obtain the following description for $G(\cB;K^0)$.

\subsubsection{Case $\fgl(n|n)$}
In this case $\nu\mapsto \lambda$ is an edge in $G(\cB;K^0)$
if and only if the diagram of $\lambda$ is obtained
from the diagram of $\nu$ by moving one symbol $\times$ along the arch originated at this symbol. 
Each vertex has exactly $n$ direct successors.

\subsubsection{Case $\fq_{2n}$ with $\cB=\cB_{1/2}$}
In this case $\nu\mapsto \lambda$ is an edge in $G(\cB;K^0)$
if and only if the diagram of $\lambda$ is obtained
from the diagram of $\nu$ by moving one symbol $\times$ along the arch originated at this symbol. Each
vertex has exactly $n$ direct successors.

\subsubsection{Case $\osp(2n+t|2n)$}
The map $\tau$ gives an isomorphism between the graphs $G(\cB;K^0)$
for $t=1$ and $t=2$. For $t=0,2$
 an edge $\nu\mapsto \lambda$ appears in $G(\cB;K^0)$ if and only if
 the diagram of $\lambda$ is  obtained
from the diagram of $\nu$ by one of the following operations:
\begin{itemize}
\item moving
one symbol $\times$ from the zero position to the farthest position
connected to the zero position;

\item moving one symbol $\times$ along the two-legged arch originated at this symbol
\end{itemize}
and, for $t=0$, the diagrams of $\lambda$ and $\nu$ do not have different
signs (for $t=2$ the diagrams do not have signs).
As a result, for $t=1,2$ each vertex has exactly $n$ direct successors;
for $t=0$ this holds for the vertices $\nu$ with $\tail\nu=0$
(observe that for $n=1$
the vertex $0$ has two  direct successors $\delta_1\pm\vareps_1$).

\subsubsection{Case $\fq_{2n+\ell}$ with $\cB=\cB_0$}\label{Kqq}
Let  $\arc(0;b',b)$ 
be the maximal three-legged arch in $\Arc(\nu)$. By~\ref{Kq}, $\nu\mapsto \lambda$ is an edge in $G(\cB;K^0)$ if and only if
the diagram of $\lambda$ is  obtained
from the diagram of $\nu$  by
moving one symbol $\times$ from a position $a$ to a free position $a'$
connected with $a$ subject to the condition $a'\not=b$; the edge
 $\nu\mapsto \lambda$ is simple if $a'\not=b'$ and is double if 
 $a'=b'$. Note that the number of
three-legged arches in
$\Arc(\nu)$ is equal to $\tail\nu$. We conclude that
 each vertex $\nu$ is the origin of $n+\tail\nu$ edges
with no double edges if $\tail\nu=0$ and a unique double edge 
if $\tail\nu>0$.

\subsubsection{}
\begin{cor}{corExt1}
\begin{enumerate}
\item
For the cases $\fg=\fgl(n|n), \osp(2n+t|2n)$ with $\cB=\cB_0$
and for $\fg=\fq_{2n}$ with $\cB=\cB_{1/2}$ the map $\pari$
gives a bipartition of  $G(\cB;K^0)$.

\item The graph 
does not have multiedges except for the case $(\fq_{2n+\ell},\cB_0)$
where the double edges appear.
\end{enumerate}
\end{cor}

\subsubsection{}\label{B+}
Fix $p\in\mathbb{Z}$ for $\fg=\fgl(n|n)$, $p\in\mathbb{N}-1/2$ for
$\cB_{1/2}$ and $p\in\mathbb{N}$ for other cases. We set
$$\cB_{>p}=\{\lambda\in\cB|\ \lambda_n>p\},\ \ \ B_+:=\{\mu\in \mathbb{N}^n|\ \mu_1>\mu_2>\ldots>\mu_n>0\}$$
and identify $\cB_{>p}$ with $B_+$ via the map
$\mu\mapsto (\mu_1-p;\mu_2-p;\ldots;\mu_n-p)$. Note that the weight diagram of $\mu\in\cB_+$ contains $\circ$ or $\times$
in each position
and the corresponding arch diagram ``does not depend on
the type of $\fg$''. By~\ref{Kgl}-- \ref{Kq} for
$\nu,\lambda\in\cB_{>p}$ the polynomials
$\hat{K}^{\lambda,\nu}(z;w)$ are the same for all types of
$\fg$ and $\cB$ (for fixed $n$). In particular, the induced subgraphs $(\cB_{>p};K^0)$ 
are isomorphic for all types of
$\fg$ and $\cB$.

Notice that $\cB_{>-1/2}=\cB_{1/2}$ 
so for each $p$ the graph $(\cB_{>p};K^0)$ is isomorphic to
$(\cB_{1/2};K^0)$.

\subsection{Graph $(\cB;\ext)$}
Recall that $\ext(\lambda;\nu)=\ext(\nu;\lambda)$ and 
$\ext(\lambda;\nu)=0$ if $\lambda\in\cB$, $\nu\not\in\cB$. Retain notation of~\ref{graphtK0}.
One has
$$s(\lambda;\nu)=n+1-\max\{i|\ \lambda_i=\nu_i\}.$$
By~\Cor{cora2}, each pair $(\lambda;\nu)$ with $\lambda\not=\nu$ is $K^i$-stable for 
any $i$.

The following 
corollary describes the graph $(\cB;\ext)$ for $(\fg,\cB)\not=(\fq_{2n+\ell}, \cB_0)$ and gives some information for the case $(\fq_{2n+\ell}, \cB_0)$.

\subsubsection{}
\begin{cor}{corExtmain} Take $\lambda\in\cB$ and $\nu\in \cB$  with 
$\nu<\lambda$. 
\begin{enumerate}
\item If $(\fg,\cB)\not=(\fq_{2n+\ell}, \cB_0)$, then
$\ext(\lambda;\nu)=k_0(\lambda;\nu)\leq 1$.
The module $\Gamma^{\fg,\fp} L_{\fp}(\lambda)$ has a semisimple
radical.

\item If $(\fg,\cB)\not=(\fq_{2n+\ell}, \cB_0)$, then
$\pari$ is a bipartition of the graph $(\cB;\ext)$.

\item
If  $(\fg,\cB)=(\fq_{2n+\ell},\cB_0)$, then
$\ext(\lambda;\nu)\leq k_0(\lambda;\nu)$. If $\lambda_n>1+\ell$, then 
\begin{itemize}
\item $\ext(\lambda;\nu)=k_0(\lambda;\nu)\leq 1$;
\item 
$k_0(\lambda;\nu)\not=0$ implies
$\pari(\nu)\not=\pari(\lambda)$;
\item the module $\Gamma^{\fg,\fp} L_{\fp}(\lambda)$ as a semisimple
radical.
\end{itemize}
\end{enumerate}
\end{cor}
\begin{proof}
For $\fg\not=\fq_{2n+\ell}$ one has $\fh=\ft$ and 
$\ext(\lambda;\lambda)=0$; for $\fq_{2n}$ one has
 $\ext(\lambda;\lambda)=0$ for  any $\lambda\in\cB_{1/2}$.
Combining Corollaries~\ref{corgraphs} and~\ref{corExt1} 
we obtain (i), (ii) and the inequality 
$\ext(\lambda;\nu)\leq k_0(\lambda;\nu)$
in (iii). Take $(\fg,\cB)=(\fq_{2n+\ell},\cB_0)$. The assumption 
$\lambda_n>1+\ell$ gives $B(\lambda)\subset \cB_{>0}$ (see~\ref{B+} for notation).
By~\ref{B+} the map $\pari$ is a bipartion of the graph $G(B(\lambda),K^0)$  
and $k_0(\lambda;\nu)=1$ for each $\nu\in B(\lambda)$. Using~\Cor{corgraphs} we 
obtain  all assertions of (iii).
\end{proof}

\subsubsection{Example}\label{exan=11}
For  $n=1$ the graphs $G(\cB;K^0)$ are given in~\ref{expn=1}.
The corresponding $\ext$-graphs, $A^{\infty}_{\infty}$
for $\fgl(1|1)$, $\osp(2|2)$, $D_{\infty}$ for $\osp(4|2)$ and $A_{\infty}$ for the rest of the cases, appear in Introduction. 
In agreement with~\Cor{corExtmain} (i)
for $(\fg,\cB)\not=(\fq_m,\cB_0)$   the $\ext$-graph can be
 obtained form $G(\cB;K^0)$ by erasing the dotted arrows and
changing $\longrightarrow$ to $\longleftrightarrow$; in this case
$\pari$ is the bipartition of the $\ext$-graph. In the remaining cases (for $n=1$)
the  $\ext$-graphs are 
$$\xymatrix{& \fq_2, \cB_0: &  &  & 0\ar[r]& {\theta}\ar[r]\ar[l] &2\theta\ar[r]\ar[l] & 3\theta\ar[r]\ar[l] 
& \ldots & \\
& \fq_3, \cB_0: & & & \theta\ar[r]& 0\ar[r]\ar[l] &2\theta\ar[r]\ar[l] &3\theta
\ar[r]\ar[l]&\ldots& \\
}
$$ 
see~\cite{MM},\cite{GNS}. 
Combining~\ref{expn=1} and~\ref{exan=11}, we conclude that
for $\fq_{2}$ the radical of  $\Gamma^{\fg,\fp} L_{\fp}(\theta)$
is an indecomposable isotypical module of length two
 with the cosocle isomorphic to $L_{\fg}(0)$, and for
$\fq_{3}$ the radical of  $\Gamma^{\fg,\fp} L_{\fp}(2\theta)$ is 
a module of length
three with the subquotients isomorphic to $L_{\fg}(0),L_{\fg}(0),L_{\fg}(\theta)$
and the cosocle isomorphic to $L_{\fg}(0)$.

\subsubsection{Remark}
Take $\fg=\fq_{2n+1}$ and $\lambda\in\cB_0$ such that $\lambda_n=1$ and
$\lambda_{n-1}>4$.
Let us show that conclusions of~\Cor{corExtmain} (iii) holds for such $\lambda$. Take $\mu$ such that
$k_0(\lambda;\mu)\not=0$. Since
$\diag\lambda=>\times\circ\circ\circ g$ for some diagram $g$
one has $\diag\mu=\overset{\times}{>}\circ \circ\circ\circ g\ $
 or $\diag\mu=>\times\circ\circ\circ f$ with $g=(f)_a^b$.
 In both cases $k_0(\lambda;\mu)=1$ and
 $\pari(\mu)\not=\pari(\lambda)$.  
 Hence  $G(B(\lambda);K^0)$ is bipartite,
 $\Gamma^{\fg,\fp} L_{\fp}(\lambda)$ has a semisimple
radical and~\Cor{corgraphs} gives $\ext(\lambda;\nu)=k_0(\lambda;\nu)$.

\subsubsection{Remark}\label{Alex}
Fix $p\in\mathbb{Z}$ for $\fg=\fgl(n|n)$, $p\in\mathbb{N}-1/2$ for
$\cB_{1/2}$ and $p\in\mathbb{N}$ for $\osp(2n+t|2n)$, $p\in\mathbb{N}_{>\ell}$
for $(\fq_{2n+\ell},\cB_0)$. Retain notation of~\ref{B+}. 
By~\Cor{corExtmain} 
$\ext(\lambda;\nu)=k_0(\lambda;\nu)$ for each $\lambda,\nu\in\cB_{>p}$ with $\lambda>\nu$. In the light of~\ref{B+}, for $\lambda\not=\nu$ the value
$\ext(\lambda;\nu)$ does not depend on
$\fg$ and $p$ (under the identification of $\cB_{>p}$ with $B_+$).
Let $\CC_+$ be the Serre subcategory  \footnote{  by Serre subcategory generated by a set of simple modules
we mean the full subcategory consisting of the modules of finite length whose all simple
subquotients lie in a given set.} of $\Fin(\fg)$
generated by $L(\lambda)$ with $\lambda\in\cB_{>p}$. By above, the graphs
$(\CC_+;\ext)$ are naturally isomorphic for all $\fg$ with $p$ as above.
For $\fgl(n|n)$ and $\osp(2n+t|2n)$ this implies the isomorphisms
between $\Ext^1$-graphs of $\CC_+$ (in these cases $\Ext^1$-graphs of $\CC_+$
have two connected components which differ by $\Pi$).

\subsection{Proof of Theorem A}\label{lambdanuP+}

Let $\tilde{\fg}$ be one of the algebras $\fgl(m|n)$, $\osp(M|2n)$ or $\fq_m$. 
We will say that weights $\lambda,\nu\in P^+(\tilde{\fg})$ have the same central character if
$L(\lambda), L(\mu)$ have the same central character.
The computation of $\ext(\lambda;\nu)$
for arbitary $\lambda,\nu\in P^+(\tilde{\fg})$ can be reduced to the case
 $\lambda,\nu\in\cB$  with the  help of
 translation functors which map a simple module in a given block
 to an isotypical semisimple module in another block of the same atypicality.
For $\tilde{\fg}\not=\fq_m$, these semisimple modules are simple and, by~\cite{GS},
  each block of atypicality $k$ in $\Fin(\fgl(m|n))$ (resp.,
in $\Fin(\osp(M|2n))$ is equivalent to 
the principal block in $\fgl(k|k)$ (resp.,  in  $\osp(2k+t|2k)$).
In particular, if $L(\mu),L(\nu)$ have the same central character, then
$\ext(\mu;\nu)=\ext(\ol{\mu};\ol{\nu})$, where
$\ol{\mu},\ol{\nu}$ are the corresponding weights in $\cB$ ($\ol{\nu}$ is
described in~\cite{GS}, Section 6). This gives Theorem A for 
$\fg=\fgl(m|n),\osp(M|2n)$ and describes the graph $(\Fin(\fg);\ext)$ in these cases.

For $\fq_m$-case the situation is more complicated, see~\cite{SICM}.

\subsubsection{Weight diagrams for $\fq_m$}\label{wtdiagqm}
For $\mu=\sum_{i=1}^m a_i\vareps_i$ denote by $\core(\mu)$ the set obtained
from $\{a_i\}_{i=1}^m$ by deleting the maximal number of
pairs satisfying $a_i+a_j=0$; for example, for $m=8$ $\core(2\vareps_1+\varesp_2-\vareps_8)=\{2;0\}$. From the description of the centre of $\cU(\fq_m)$ obtained in~\cite{Sq}, it follows  that $L(\lambda),L(\mu)$ have the same central character if
and only if $\core(\lambda)=\core(\mu)$.
We set $\ell(\lambda):=0$ if $0\not\in\core(\lambda)$  and $\ell(\lambda):=1$ otherwise.

The weight diagram
for $\mu=\sum_{i=1}^m a_i\vareps_i\in P^+(\fq_M)$ is constructed by the following procedure: we put  $>$ (resp., $<$) to the $p$th position if $a_i=p$ (resp., $a_i=-p$)
for some $i$, add $\circ$ to all empty positions  and then glue each pair $>,<$ and each pair $>,>$ (which could occur only at the zero position) to one symbol $\times$. For instance 
$$\begin{array}{lclclc}
\mu=\vareps_1-\vareps_3-3\vareps_4 &  & \nu=3\vareps_1-\vareps_4 & &
\lambda=4\vareps_1 +2\vareps_2-\vareps_3-2\vareps_4\\
\diag\mu=>\times\circ< & & \diag\nu=\times<\circ> & &\diag\lambda=\circ<\times>.
\end{array}$$
If $\mu\in\cB$ the resulting diagram coincides with the
diagram constructed in~\ref{wtdiag}.

The symbols $>,<$ are called {\em core symbols}. By above, $\lambda,\mu$
have the same central character if
and only if all core symbols in their diagrams occupy the same positions
(for instance, in the above example $\nu$ and $\lambda$ have the same central character).

In this paper we define the atypicality of the weight to be the number of $\times$ in the diagram. By contrast, in~\cite{GqDS} the atypicality of the weight is defined as  the number of $\times$ in the diagram if the diagram does not have $>$ 
at the zero position and is equal to the number of $\times$ plus $\frac{1}{2}$ if
 the diagram has  $>$ 
at the zero position. In~\cite{GqDS} the symbol $>$ at the zero position is not considered as a core symbol; note that for this definition it is still true that $\lambda,\mu$
have the same central character if
and only if all core symbols in their diagrams occupy the same positions.

Let $\eta\in P^+(\fq_M)$ be a weight of atypicality $n>0$.
We denote by $\ol{\eta}$ the weight in $P^+(\fq_{2n+\ell})$ with the weight diagram which is obtained from $\diag\eta$ by erasing all core symbols at the non-zero positions. For  the above example we have
$$\begin{array}{lclclc}
\diag\ol{\mu}=>\times & & \diag\ol{\nu}=\times  & &\diag\ol{\lambda}=\circ\times\\
\ol{\mu}=\vareps_1-\vareps_3&  & \ol{\nu}=0 & &
\ol{\lambda}=\vareps_1-\vareps_2.
\end{array}$$
Note that $\ol{\eta}\in \cB_0$ if $\eta$ is integral and $\ol{\eta}\in\cB_{1/2}$
if $\eta$ is half-integral (for example, the weight $\eta=\frac{3}{2}\vareps_1-
\frac{1}{2}\vareps_2-\frac{3}{2}\vareps_3$ has the diagram $>\times$, so
$\diag\ol{\eta}=\times$ and
$\ol{\eta}=\frac{1}{2}\vareps_1-
\frac{1}{2}\vareps_2$).

\subsubsection{}
\begin{prop}{propqm}
For any $\eta,\zeta\in P^+(\fq_m)$ one has $\ext(\eta;\zeta)=\ext(\ol{\eta},\ol{\zeta})$
if $\eta,\zeta$ have the same central character.
\end{prop}
\subsubsection{Outline fo the proof}\label{pfpropqm}
Take $\tilde{\fg}=\fq_m$ with a triangular decomposition $\tilde{\fg}=\tilde{\fn}_+\oplus\tilde{\fh}\oplus \tilde{\fn}_-$.
A weight $\mu\in P^+(\tilde{\fg})$ 
is called {\em stable} if all  symbols $\times$ preceed all core symbols  with non-zero coordinates (in the above example $\mu,\nu$ are stable weights and $\lambda$ is not stable). 

Let  $\eta$ be a stable weight of atypicality $n$. Then
$\tilde{\fg}$ contains a subalgebra $\fg\cong\fq_{2n+\ell(\eta)}$
with a compatible triangular decomposition such that 
the restriction  of $\eta$ to the Cartan subalgebra of $\fg_0$ equal to
$\ol{\eta}$ (in the above example, for $\mu$ one has
$\fg\cong\fq_3$ corresponding to $\vareps_1,\vareps_2,\varesp_3$
and for $\nu$ one has
$\fg\cong\fq_2$ corresponding to $\vareps_2,\varesp_3$).
By~\cite{PS2}, Corollary 1 for $\fp:=\fg+\tilde{\fb}$ 
one has $\Gamma^{\tilde{\fg},\fp}_0 L_{\fp}(\eta)=L_{\tilde{\fg}}(\eta)$
and  $\Gamma^{\tilde{\fg},\fp}_i L_{\fp}(\eta)=0$ for $i>0$. Combining~\ref{lemExtVerma}, \ref{lem332} and~\ref{corbox}  we obtain 
 $\ext(\eta;\zeta)=\ext(\ol{\eta};\ol{\zeta})$ if $\eta,\zeta$
 are stable weights with the same central character.

The general case can be reduced to
 the stable case with the help of translation functors 
described in~\cite{Br}. A translation 
functor which preserves
the degree of atypicality and the value of $\ell(\eta)$   transforms $L(\eta)$ 
to $L(\eta')\oplus \Pi L(\eta')$, where  $\diag\eta'$ is obtained from 
$\diag\eta$ by permuting two neighboring symbols at non-zero positions
if exactly one of these symbols is a core symbol:
for instance, for $\lambda$ as above we can obtain $\lambda'$s with the diagrams
$\circ<>\times$ and $\circ\times<>$ (the last diagram is stable).
Note that $\ol{\eta}=\ol{\eta}'$.
Using these fucntors we can transform any two simple modules
$L(\eta)$, $L(\zeta)$ with the same central character 
to the modules $L(\eta')^{\oplus r}\oplus \Pi L(\eta')^{\oplus r}$, $L(\zeta')^{\oplus r}\oplus \Pi L(\zeta')^{\oplus r}$,
where $\eta',\zeta'$ are stable weights with the same central character
and $\ol{\eta}=\ol{\eta}'$, $\ol{\zeta}=\ol{\zeta}'$
(the diagrams of $\eta'$ and $\zeta'$
 are stable diagrams obtained from the diagrams of $\eta$ and $\zeta$ by 
moving all core symbols from the  non-zero position ``far enough'' to the right).
It is not hard to show  that these functors map a module from
$\cN(\eta;\zeta;m)$ to a module of the form $M\oplus \Pi M$, where
$M$ is a direct sum of  $r$ modules
from $\cN(\eta',\zeta';m)$. This gives $\ext(\eta;\zeta)\leq \ext(\eta';\zeta')$.
Using the same set of functors we can transform $L(\eta'),L(\zeta')$ 
 to the modules $L(\eta)^{\oplus r}\oplus \Pi L(\eta)^{\oplus r}$, $L(\zeta)^{\oplus r}\oplus \Pi L(\zeta)^{\oplus r}$; this implies
$\ext(\eta';\zeta')\leq \ext(\eta;\zeta)$.
Since $\eta',\zeta'$ are stable we obtain 
$\ext(\eta;\zeta)=\ext(\eta';\zeta')=
\ext(\ol{\eta};\ol{\zeta})$ as required. \qed

\subsubsection{}
The arch diagrams for an arbitrary $\lambda\in P^+(\fq_M)$ are constructed 
in the same way as the arch diagrams for $\lambda\in\cB$: starting from the rightmost
symbol $\times$  in the weight diagram $\diag(\lambda)$ 
we connect each symbol $\times$ at the non-zero position with the next free symbol $\circ$, then 
each symbol $\times$ at the zero position  with the next two free symbols $\circ$ and then add wobbly arch if there is $>$ at the zero position. There is a natural bijection
between the arches in
$\Arc(\lambda)$ and $\Arc(\ol{\lambda})$. 
\subsubsection{}
\begin{cor}{corThm1}
For $\lambda>\nu\in P^+(\fq_m)$ one has
\begin{enumerate}
\item $\ext(\lambda;\nu)\leq 2$;
\item if $\ext(\lambda;\nu)\not=0$, then 
 $\diag{\lambda}$  can be obtained from 
 $\diag\nu$  by moving one  symbol $\times$ along the arch in $\Arc(\nu)$;
 \item if  $\diag\ol{\lambda}$ does not have $\times$ at the position  
 $0, 1, 1+\ell(\lambda)$,
 then $\ext(\lambda;\nu)=1$ if 
  $\diag{\lambda}$  can be obtained from 
 $\diag\nu$  by moving one  symbol $\times$ along the arch
 and $\ext(\lambda;\nu)=0$ otherwise.
\end{enumerate}
\end{cor}

\subsubsection{}
The above Corollary implies
Theorem A for $\fq_m$ and gives a description of the graph 
$(\Fin(\fq_m)_{1/2},\exp)$, where
$\Fin(\fq_m)_{1/2}$ is the full subcategory consisting of the modules
with half-integral weights. Note that $\pari(\ol{\lambda})$ is a biparition
of this graph.

 By above,
 $\ext_{\fq_3}(2\theta,\theta)=0$, so
 the converse of (ii) does not hold (in this case $\diag 2\theta=>\times$, so $\times$ occurs at the position $1+\ell(2\theta)=2$).

\subsection{Properties (Dex1), (Dex2)}\label{Dex12}
Consider the case $\fg=\fgl(m|n),\osp(M|2n)$.
The map $\pari(\ol{\lambda})$ is a bipartition of
the graph $(\Fin(\fg);\exp)$. The map $\Irr(\Fin(\fg))\to\mathbb{Z}_2$
given by  $L(\lambda),\Pi L(\lambda)\mapsto \pari(\ol{\lambda})$ 
satisfies (Dex1), but
does not satisfy (Dex2). A map satisfying (Dex1) and (Dex2) can be 
constructed using a certain
decompostion $\Fin(\fg)=\cF\oplus \Pi \cF$ (for atypical modules $N$
we take $N\in\cF$ if
$N_{\ol{i}}=\sum_{\mu: p(\mu)=\ol{i}} N_{\mu}$, where
$p(\mu)$ is given by $p(\vareps_i)=\ol{0}$, $p(\delta_j)=\ol{1}$).
Taking $\pari(L(\lambda)):=\pari(\ol{\lambda})$ for $L(\lambda)\in\cF$
and $\pari(L(\lambda)):=\pari(\ol{\lambda})+\ol{1}$ for $L(\lambda)\in \Pi\cF$
we obtain a map satisfying (Dex1) and (Dex2), see~\cite{HW},\cite{GH} for details.

For the case $(\fq;\frac{1}{2})$ the map $\pari(\ol{\lambda})$ is a bipartition of the graph $(\Fin(\fq_m)_{1/2};\exp)$ (where $\Fin(\fg)_{1/2}$ is
the full subcategory of $\Fin(\fq_m)$ consisting of the modules
with half-integral weights); the map
$\Irr(\Fin(\fq_m)_{1/2})\to\mathbb{Z}_2$
given by  $L(\lambda),\Pi L(\lambda)\mapsto \pari(\ol{\lambda})$ 
satisfies (Dex1). 

For the remaining case $(\fq;\CC)$ one has $\fg:=\fq_m$; we denote by $\CC$ the Serre subcategory generated
by $L(\lambda),\Pi L(\lambda)$  with $\ol{\lambda}$ satisfying the assumption
of~\Cor{corThm1} (iii).
By above, $\pari(\ol{\lambda})$ is a bipartition of the graph $(\CC;\exp)$ and the map
$\Irr(\CC)\to\mathbb{Z}_2$
given by  $L(\lambda),\Pi L(\lambda)\mapsto \pari(\ol{\lambda})$ 
satisfies (Dex1).

\subsection{Remark}\label{Extext}
If $\CC$ is a full $\Pi$-invariant subcategory of $\CO(\fg)$ and
$\fg$ is a Kac-Moody superalgebra, then
the $\Ext^1$-graph of $\CC$ is a disjoint union of two copies
of the graph $(\CC;\ext)$. In particular, if $\ext$-graphs of $\CC\subset \CO(\fg)$
and $\CC'\subset \CO(\fg')$ are isomorphic, then
$\Ext^1$-graphs of $\CC$ and $\CC'$ are isomorphic, see examples in~\ref{Alex}.
This does not hold for $\fq_m$: for instance, the half-integral principal
block in $\Fin(\fq_2)$ and the intergal principal blocks in $\Fin(\fq_2),\Fin(\fq_3)$
have isomorphic $\ext$-graphs and different $\Ext^1$-graphs (see~\cite{MM},\cite{GNS}).

\subsubsection{}
Take $\fg=\fq_m$.   Fix a central character $\chi$
and let $\CC_{\chi}$ be the corresponding Serre subcategory of $\Fin(\fq_m)$.
We assume that $\CC_{\chi}\not=0$ and denote by $(CC_{\chi},\Ext^1)$ the $\Ext^1$-graph 
of $\CC_{\chi}$.

By above, the set $\core(\lambda)$ is the same for all $L(\lambda)\in\CC_{\chi}$.
We denote this set by $\core(\chi)$.
We say that $\chi$ is $\Pi$-invariant if $\core(\chi)\setminus\{0\}$
contains an odd number of elements; in this case 
each $L(\lambda)\in \CC_{\chi}$ 
 is $\Pi$ -invariant. If $\chi$ is not $\Pi$-invariant, then 
each $L(\lambda)\in \CC_{\chi}$ 
 is not $\Pi$ -invariant.

 We say that $\chi$ is intergal (resp., half-integral) if $\core(\chi)$ contains
 an intergal (half-integral) number. 
 If $\chi\not=\chi_0$ is atypical, then $\chi$ is either integral or half-integral and
the graph $(\CC_{\chi};\ext)$ is connected. If
$\chi=\chi_0$ and $m$ is even, the graph $(\CC_{\chi};\ext)$ 
has two connected components $(\cB_0;\ext)$ and $(\cB_{1/2};\ext)$.

If  $\chi$ is $\Pi$-invariant, the graph $(\CC_{\chi},\Ext^1)$ can be 
obtained from $(\ext;\CC_{\chi})$ by adding the loops around each vertex 
$\lambda$ with $0\in\{\lambda_i\}_{i=1}^m$, see~\cite{GNS}, Theorem 3.1.
In particular, $(\CC_{\chi},\Ext^1)=(\ext;\CC_{\chi})$
if $\chi$ is $\Pi$-invariant and half-integral.

Consider the case when  $\chi$ is not $\Pi$-invariant.
The vertices of $(\CC_{\chi};\Ext^1)$ are of the forms
$(\nu;i)$, where $\nu\in\CC_{\chi}, i\in\mathbb{Z}_2$.
By Theorem 3.1 in~\cite{GNS} the graph $(\CC_{\chi};\Ext^1)$  does not have loops and
the vertices $(\nu,i),(\nu,i+1)$ are connected by a unique edge
$\longleftrightarrow$ if $0\in\{\lambda_i\}_{i=1}^m$; otherwise
these vertices are not connected. Consider the edges of the form
$(\nu,i)\leftrightarrow (\lambda,j)$. 
Each edge $\nu\leftrightarrow \lambda$ corresponds to fours edges
 $(\nu,0)\leftrightarrow (\lambda,j)$, $(\lambda,j)\leftrightarrow (\nu,i)$
and  $(\nu,1)\leftrightarrow (\lambda,j+1)$, 
$(\lambda,j+1)\leftrightarrow (\nu,i+1)$ for some $i,j$.
By~\cite{Frisk}, 
$$\dim\Ext^1(L(\lambda),L(\nu))=\dim \Ext^1(L(\nu),\Pi^{\tail(\ol{\lambda};\ol{\nu})} L(\lambda))$$ which implies $i=0$ if $\tail(\ol{\lambda};\ol{\nu})$ is even
and $i=1$ if $\tail(\ol{\lambda};\ol{\nu})$ is odd. 
In particular, if $\chi$ is not $\Pi$-invariant and half-integral, then
in the graph $(\CC_{\chi};\Ext^1)$ the vertices 
$(\nu,i),(\nu,i+1)$ are not connected and all edges are of the form $\longleftrightarrow$.

Unfortunately, the above information is not sufficient for a description
of $(\CC_{\chi};\Ext^1)$ for atypicality greater than one (the graphs for atypicality
one  were described in~\cite{MM},\cite{GNS}).


\end{document}